\documentclass[11pt, a4paper]{amsart}

\usepackage{geometry}
\usepackage{enumerate, mathtools}
\usepackage{hyperref}
\hypersetup{
    colorlinks=true,
    linkcolor=blue,
    filecolor=blue,      
    urlcolor=blue,
    citecolor=blue,
}
\usepackage{amsmath}
\usepackage{amsthm}
\usepackage{amssymb}
\usepackage{xcolor}
\usepackage{url}
\usepackage{enumitem}
\usepackage[capitalize, noabbrev]{cleveref}
\usepackage{longtable}
\usepackage{bm}
\usepackage{xypic}
\usepackage{soul}
\usepackage{stmaryrd}
\usepackage{tabularray}
\usepackage{subcaption}
\usepackage{graphicx}

\usepackage[utf8]{inputenc}
\usepackage{pgfplots}
\DeclareUnicodeCharacter{2212}{−}
\usepgfplotslibrary{groupplots,dateplot}
\usetikzlibrary{patterns,shapes.arrows}
\pgfplotsset{compat=newest}

\usepackage{algorithm}
\usepackage[noend]{algpseudocode}

\numberwithin{equation}{section}

\newtheorem{thmabc}{Theorem}

\newtheorem{corabc}[thmabc]{Corollary}

\newtheorem{thm}{Theorem}

\numberwithin{thm}{section}

\newtheorem*{thm*}{Theorem}
\newtheorem*{con*}{Conjecture}
\newtheorem{lem}[thm]{Lemma}
\newtheorem{prop}[thm]{Proposition}
\newtheorem{coro}[thm]{Corollary}
\newtheorem{lemma}[thm]{Lemma}
\newtheorem{conj}[thm]{Conjecture}

\theoremstyle{definition}
\newtheorem{defn}[thm]{Definition}
\newtheorem{remark}[thm]{Remark}
\newtheorem{ex}[thm]{Example}
\newtheorem{quest}[thm]{Question}

\DeclareMathOperator{\End}{End}
\DeclareMathOperator{\Hom}{Hom}
\DeclareMathOperator{\Aut}{Aut}
\DeclareMathOperator{\GL}{GL}

\DeclareMathOperator{\diag}{diag}
\DeclareMathOperator{\Gal}{Gal}

\DeclareMathOperator{\im}{im}

\DeclareMathOperator{\tuB}{B}
\DeclareMathOperator{\tuM}{M}
\DeclareMathOperator{\tuN}{N}
\DeclareMathOperator{\GP}{G}
\DeclareMathOperator{\LA}{\mathfrak{g}}
\DeclareMathOperator{\Mat}{Mat}
\DeclareMathOperator{\id}{id}

\DeclareMathOperator{\tp}{t}
\DeclareMathOperator{\ZG}{Z}

\DeclareMathOperator{\Pf}{Pf}

\DeclareMathOperator{\Cent}{Cent}
\DeclareMathOperator{\Adj}{Adj}
\DeclareMathOperator{\Div}{Div}
\DeclareMathOperator{\Pic}{Pic}

\newcommand{\Q}{\mathbb{Q}}

\newcommand{\F}{\mathbb{F}}
\newcommand{\Z}{\mathbb{Z}}

\renewcommand{\phi}{\varphi}
\renewcommand{\leq}{\leqslant}
\renewcommand{\geq}{\geqslant}
\renewcommand{\epsilon}{\varepsilon}
\renewcommand{\P}{\mathbb{P}}
\newcommand{\alg}{K}
\newcommand{\ff}{F}

\newcommand{\pseudo}{\Psi\mathrm{Isom}}
\newcommand{\spseudo}{\mathrm{S}\Psi\mathrm{I}}
\newcommand{\J}[2]{J_{#1,#2}}
\newcommand{\charac}{\mathrm{char}}
\newcommand{\ti}[1]{\mathcal{T}(#1)}

\newcommand{\cor}[1]{\mathcal{#1}}
\newcommand{\ps}[1]{\llbracket #1 \rrbracket}

\allowdisplaybreaks

\begin{document}

\title{
Smooth cuboids in group theory
}

\author[Joshua Maglione]{Joshua Maglione}
\address[Joshua Maglione]{
University of Galway,  School of Mathematical and Statistical Sciences
}
\email{joshua.maglione@universityofgalway.ie}

\author[Mima Stanojkovski]{Mima Stanojkovski}
\address[Mima Stanojkovski]{
Universit\`a di Trento, Dipartimento di Matematica
}
\email{mima.stanojkovski@unitn.it}

\makeatletter
\@namedef{subjclassname@2020}{
 \textup{2020} Mathematics Subject Classification}
\makeatother


\begin{abstract}
A smooth cuboid can be identified with a $3\times 3$ matrix of linear forms in $3$ variables,  with coefficients in a field $\alg$, whose determinant describes a smooth cubic in the projective plane.  To each such matrix one can associate a group scheme over $\alg$. We produce isomorphism invariants of these groups in terms of their \emph{adjoint algebras},  which also give information on the number of their maximal abelian subgroups. Moreover, when $\alg$ is finite, we give a characterization of the isomorphism types of the groups in terms of isomorphisms of elliptic curves and also describe their automorphism groups. 
We conclude by applying our results to the determination of the automorphism groups and isomorphism testing of finite $p$-groups of class $2$ and exponent $p$ arising in this way.
\end{abstract}

\maketitle

\setcounter{tocdepth}{1}
\tableofcontents

\section{Introduction}

The Baer correspondence is a classical way of associating an alternating
bilinear map with a nilpotent group or a nilpotent Lie algebra over some field $\alg$. 
In the case of groups of prime power order, this is also referred to as the Lazard correspondence and allows one to study groups in the (often easier) context of Lie algebras.
As $\alg$-bilinear maps can be represented as matrices of linear forms, the study of groups arising from the Baer correspondence, moreover, affords an additional geometric point of view.
In this paper, our focus is on a
systematic study of the groups, which we call
\emph{$E$-groups}, obtained by inputting matrices of linear forms,
whose determinant defines an elliptic curve in $\P^2_{\alg}$. This work also expands on both the results and techniques from~\cite{stanojkovski2019hessian}.

\subsection{Notation}
Throughout $p$ denotes an odd prime number, and $G$ a $p$-group.
Moreover, $\alg$ denotes an arbitrary field and $\ff$ a finite field of
characteristic $p$ and cardinality $q$. Let $n$ and $d$ be positive integers. For a vector ${\bm y}=(y_1,\ldots,y_d)$, we denote by  $\alg[\bm{y}]_1$ the vector space of linear homogeneous polynomials with coefficients in $\alg$. We write $\Mat_n(\alg[\bm{y}]_1)$ for the $n\times n$ matrices of linear forms in $y_1,\ldots,y_d$ with coefficients in $\alg$. 
\subsection{Baer correspondence and Pfaffians}\label{sec:intro-Baer}
Given a skew-symmetric
matrix of linear forms $\tuB\in\Mat_{2n}(\ff[{\bm y}]_1)$, the Baer
correspondence associates $\tuB$ with a $p$-group $G=\GP_{\tuB}(\ff)$ of exponent $p$ and class at most $2$ with underlying set $\ff^{2n+d}$. A \emph{non-degenerate} matrix $\tuB$ completely prescribes the commutator relations on $G$ and ensures that the bilinear map 
\begin{equation}\label{eq:commap}
t_G:G/\ZG(G)\times G/\ZG(G)\longrightarrow G'=[G,G]
\end{equation}
induced by the commutator map on $G$ is $\ff$-bilinear (so not just
$\F_p$-bilinear). On the other hand, as we explain  below, $\ff$ and $\tuB$ can
also be recovered  from $G$.

Given a finite $p$-group $G$ of class $2$ and exponent $p$, one can
construct a skew-symmetric matrix of linear forms $\tuB_G $ over $\F_p$ from the
commutator map of $G$, for example in terms of a minimal generating set of $G$.
Moreover, if $\ff$ is an extension of $\F_p$ over which the
map in \eqref{eq:commap} is $\ff$-bilinear, then we can express $\tuB_G$ as a matrix $\tuB=\tuB_\ff$ of linear forms with coefficients in $\ff$.
If $\tuB$ has an odd number of rows and columns, then
$\det(\tuB)=0$; otherwise, there exists a homogeneous polynomial
$\Pf(\tuB)\in\ff[y_1,\dots, y_d]$ of degree $n$, called the \emph{Pfaffian},
such that $\det(\tuB) = \Pf(\tuB)^2$. The equation $\Pf(\tuB)=0$ defines a
projective, degree $n$ hypersurface in $\P^{d-1}_{\ff}$, also called a
\emph{linear determinantal hypersurface}. There are now many examples in the
literature, cf.\ \cite[Sec.~1.5.4]{stanojkovski2019hessian}, describing how the intrinsic geometry 
of $\tuB$ strongly influences many invariants of
$\GP_{\tuB}(\ff)$, like the number of conjugacy classes \cite{BostonIsaacs04,O'BrienVoll/15,Rossmann/19,Rossmann/22,RV/2019}, faithful dimensions \cite{BMS/17},
automorphism group sizes \cite{dS+VL,stanojkovski2019hessian,VL/18}, and the number of immediate descendants \cite{dS+VL,Lee/16}. We look at the Baer correspondence in more detail in \cref{sec:baer3}.

\subsection{Elliptic groups}

The focus of this paper is, in the language of \cref{sec:intro-Baer}, to study
groups $\GP_{\tuB}(\ff)$ with $d=n=3$ and $\Pf(\tuB)=0$ defining an elliptic
curve. Moreover, we are concerned with isomorphisms and automorphisms of abstract groups.

\begin{defn}\label{defn:EIH-groups}
An \emph{elliptic group} (abbreviated to {\em $E$-group}) is a group $G$ that is isomorphic to $\GP_{\tuB}(\alg)$ where $\tuB$ is a skew-symmetric matrix of linear forms over a field $\alg$ such that
$\Pf(\tuB)=0$ defines an elliptic
curve in $\P^2_{\alg}$. If $\alg$ is finite of  characteristic $p$, the group $G$ is called an \emph{elliptic $p$-group}.
\end{defn}

From \cref{defn:EIH-groups} it follows that, if $G\cong \GP_{\tuB}(\ff)$ is an $E$-group, then the matrix $\tuB$ belongs to
$\Mat_6(\ff[y_1,y_2,y_3]_1)$ and 
$|G|=q^9$.

\subsection{Elliptic groups and variations over the primes}\label{sec:PORC}

In the study of $p$-groups, the groups arising from the Baer correspondence are
sometimes specializations $\GP_{\tuB}(\mathfrak{O}_k/\mathfrak{p})$ to a quotient ring or field of global unipotent group schemes $\GP_{\tuB}$
over the ring of integers $\mathfrak{O}_k$ of some number field $k$. In this respect,
it is natural to study how the properties of the group
$\GP_{\tuB}(\mathfrak{O}_k/\mathfrak{p})$ vary with $\mathfrak{p}$. Significant
concepts in this regard are those of quasipolynomiality (also called
\emph{polynomiality on residue classes}, PORC) \cite{HIgman1} and
\emph{polynomiality on Frobenius sets} (i.e.\ polynomiality on finite Boolean combinations of sets
of primes defined by the solvability of polynomial congruences; cf.~\cite[Sec.~1.5.2]{stanojkovski2019hessian}).  Of particular relevance is Higman's PORC
conjecture~\cite{Higman2} about the function $\{p \in\Z \mid p \textup{ prime}\}\rightarrow\Z$ enumerating the isomorphism
classes of groups of order $p^n$ for fixed $n$.

We recall that a function $f:\{p \in\Z \mid p \textup{ prime}\}\rightarrow\Z$ is \emph{quasipolynomial} if there exists a positive integer $N$ and polynomials $f_0,\ldots,f_{N-1}\in\Z[x]$ such that 
$$f(p)=f_i(p) \ \ \textup{whenever}\ \ p\equiv i \bmod N.$$
In the context of quasipolynomiality, elliptic groups are for instance employed
in \cite{dS+VL} to construct a family of $p$-groups where the number of
isomorphism classes is not a quasipolynomial in $p$. More specifically, the
elliptic group scheme $\GP$ is defined over $\Z$, and the family arises as a
collection of central extensions of $\GP(\F_p)$.  In earlier work, du Sautoy
\cite{duSautoy/01} showed that the number of subgroups, resp.\ normal subgroups, of index $p^3$ of
$\GP(\Z)$, for almost all primes $p$, depends on the number of
$\F_p$-points on $E$, written $|E(\F_p)|$. Du
Sautoy~\cite{duSautoy/02} extended this to show that, for a parametrized family $(E_{D})_D$ of elliptic curves
given in short Weierstrass form, (infinitely many terms of) the subgroup
zeta functions of a resulting parametrized family $(\GP_D)_D$ of $E$-groups
depend on $|E_D(\F_p)|$.  In the normal subgroup case, these zeta functions were made explicit by Voll \cite{Voll/04} as an application of more general results on smooth curves in the plane, which was further generalized in \cite{Voll/05} to smooth projective hypersurfaces with no lines.
Recently Voll and the second author~\cite{stanojkovski2019hessian} showed that
the automorphism group sizes of a class of elliptic $p$-groups are multiples of the number of $3$-torsion points of the corresponding curves. For more context, not only involving
elliptic groups, we refer to
\cite[Sec.~1.5.2]{stanojkovski2019hessian} and the references therein.

\subsection{Groups from points on curves}
\label{sec:G_EP}

As we explain in \cref{sec:determinantal}, there is a straightforward way to
construct examples of elliptic groups from triples $(\alg,E,P)$ where $\alg$ is
a field, $E$ is an elliptic curve given by a short Weierstrass equation 
\begin{equation}\label{eq:intro-wf}
y^2=x^3+ax+b \textup{ with } a,b\in \alg
\end{equation} 
and $P=(\lambda, \mu)$ is a point in $E(\alg)$. In this case, the matrix is defined as follows:
\begin{equation}\label{eqn:J_EP}
\tuB_{E,P}=\begin{pmatrix}
0 & \J{E}{P} \\
-\J{E}{P}^{\tp} & 0
\end{pmatrix} \textup{ where } \J{E}{P}=\begin{pmatrix}
y_1-\lambda y_3 & y_2-\mu y_3 & 0 \\
y_2+\mu y_3 & \lambda y_1 +(a+\lambda^2)y_3 & y_1 \\
0 & y_1 & -y_3 
\end{pmatrix}.
\end{equation}
The matrix $\J{E}{P}$ is a particular instance of a \emph{smooth  cuboid}: indeed this $3\times 3$ matrix of linear forms in $3$ variables can be interpreted as a $(3,3,3)$ \emph{cuboid} as in \cite{Ng}, and it is not difficult to see that by homogenizing the \emph{smooth} curve \eqref{eq:intro-wf}, one
recovers precisely $\Pf(\tuB_{E,P})=0$. For simplicity, we denote the group
$\GP_{\tuB_{E,P}}(\alg)$ by $\GP_{E,P}(\alg)$.

\subsection{The contributions of this paper}
\label{sec:main-results}

The main results of this paper are Theorems \ref{mainthm:sat}, \ref{mainthm:aut}, \ref{mainthm:adj},  and \ref{mainthm:iso}. The first three results are proven in \cref{sec:ABC}, while the fourth is given in \cref{sec:algs} together with the necessary computational conventions.

If $E$ is an elliptic curve, we indicate by $\cor{O}$
its identity element. If $E$ is given by a Weierstrass equation as in
\eqref{eq:intro-wf}, then, in projective coordinates,  one has $\cor{O}=(0:1:0)$.
If $E$ and $E'$ are elliptic curves in the plane with identity elements
$\cor{O}$ and $\cor{O}'$ respectively, then an isomorphism $E\rightarrow E'$ is
an isomorphism of projective varieties mapping $\cor{O}$ to $\cor{O}'$. The
automorphism group of the elliptic curve $E$ is denoted $\Aut_{\cor{O}}(E)$, and
its isomorphism type depends only on the $j$-invariant $j(E)$ of $E$.  Moreover, if $P$ is a
point on $E$, then we write $\Aut_{\cor{O}}(E)\cdot P$ for the orbit of $P$ under
the action of $\Aut_{\cor{O}}(E)$ on $E$. For a positive integer $n$, the
$n$-torsion subgroup of $E$ is denoted by $E[n]$. If $\sigma\in\Gal(\ff/\F_p)$
and $E$ is an elliptic curve defined by $f=0$ over $\ff$, then $\sigma(E)$ is
defined by the polynomial $\sigma(f)$, where $\sigma$ acts on the coefficients of $f$. Moreover, if $P=(a:b:c)\in E(\ff)$ then
$\sigma(P)\in \sigma(E)(\ff)$ is defined by
$\sigma(P)=(\sigma(a):\sigma(b):\sigma(c))$.

The following two theorems greatly generalize Theorem~1.1 from
\cite{stanojkovski2019hessian}.

\begin{thmabc}\label{mainthm:sat}
 Let $\ff$ be a finite field with $\charac(\ff)=p\geq 5$.   Let, moreover, $E$ and $E'$ be elliptic curves in $\P^2_{\ff}$ given by Weierstrass
    equations, and let $P\in E(\ff)\setminus\{\cor{O}\}$ and $P'\in
    E'(\ff)\setminus\{\cor{O}'\}$.  Then the
    following are equivalent.
    \begin{enumerate}[label=$(\arabic*)$]
        \item The groups $\GP_{E,P}(\ff)$ and $\GP_{E',P'}(\ff)$ are
        isomorphic.
        \item There exist $\sigma\in \Gal(\ff/\F_p)$ and an isomorphism
        $\varphi:E'\rightarrow \sigma(E)$ of elliptic curves such that
        $\varphi(P')=\sigma(P)$.
    \end{enumerate}
\end{thmabc}

\begin{thmabc}\label{mainthm:aut}
 Let $\ff$ be a field with $\charac(\ff)=p\geq 5$ and cardinality $p^e$.   Let, moreover,  $E$ be an elliptic curve in $\P^2_{\ff}$ given by a short Weierstrass equation,
    and let $P\in
    E(\ff)\setminus\{\cor{O}\}$.  Then there exists a subgroup $S$ of $\Gal(\ff/\F_p)$ such that the following holds:
    \[
        \dfrac{|\Aut(\GP_{E,P}(\ff))|}{{
        p^{18 e^2}}}=|S|\cdot|E[3](\ff)|\cdot \frac{|\Aut_{\cor{O}}(E)|}{|\Aut_{\cor{O}}(E)\cdot P|}\cdot\begin{cases}
        |\GL_2(\ff)| & \textup{if } P\in E[2](\ff),\\
        2(p^e-1)^2 & \textup{otherwise}. 
        \end{cases}
    \]
\end{thmabc}

For the precise definition of the subgroup $S$ from \cref{mainthm:aut} we refer the reader to \eqref{Gal_t}; see also \cref{prop:reduction} and, for an example, \cref{rmk:computing-galois}. For the groups $\GP_{E,P}(\ff)$,  a generator of the subgroup $S$ can be computed efficiently.  This is not necessarily the case for groups $\GP_{\tuB}(\ff)$, where $\tuB$ is arbitrary. 

Combining \cref{mainthm:aut} with \cite[Th.~2.2.1]{Weinstein/16} and the classification from \cite{BanPal/12}, one obtains that, although the function in~\eqref{eqn:counting-function}
counting the number of automorphisms of the family $(\GP_{E,P}(\F_p))_p$ is
polynomial on Frobenius sets, it is almost never quasipolynomial; cf.\ \cref{cor:POFS} and \cref{rmk:Duke} for some more detail.

\begin{corabc}\label{cor:POFS}
Let $E$ be an elliptic curve given by the Weierstrass equation 
\[y^2=x^3+ax+b, \quad a,b\in \Q\]
    and $P\in
    E(\Q)\setminus\{\cor{O}\}$. Over the set of primes for which $E$ has good reduction, the function 
    \begin{align}\label{eqn:counting-function}
        p &\longmapsto |\Aut(\GP_{E,P}(\F_p))|    
    \end{align}
 is polynomial on Frobenius sets and is quasipolynomial precisely in the following cases:
    \begin{enumerate}[label=$(\arabic*)$]
    \item $a=0$ and there exists $\beta\in\Q^\times$ such that $b=2\beta^3$,
    \item $b\neq 0$ and there exist $h, \ell\in\Q^{\times}$ such that 
    \[
		h^3=-16(4a^3+27b^2) \textup{ and }
		\ell^2= (-h-4a)/3     
    \]
    and one of the following holds:
    \begin{enumerate}[label=$(\alph*)$]
    \item there exist $\alpha, \beta, m\in\Q$ with $\beta\neq 0$ that satisfy:
    \[
m^2=\alpha^2+3\beta^2, \quad a=-3m^2+6\beta m, \quad b = 2\alpha^3+12\alpha\beta^2-6\alpha\beta m;
    \]
    \item for $\gamma=-\ell^2-4a-8b/\ell$, the element
  $
-\ell^3-8a\ell-16b+(-\ell^2+4b/\ell)\sqrt{\gamma}$ is a square in  the splitting field of $(x^2-\gamma)(x^3-1)$.
    \end{enumerate}
    \end{enumerate}
\end{corabc}

\begin{remark}\label{rmk:Duke}
The case distinction in \cref{cor:POFS} comes from the explicit description, given in \cite{BanPal/12}, of curves $E$ in short Weierstrass form for which the Galois group of 
$$\Q(E[3])=\Q(\{x,y : (x,y)\in E[3](\Q^{\textup{sep}})\})$$
is abelian.  Thanks to \cite[Th.~2.2.1]{Weinstein/16}, i.e.\ the Abelian Polynomial Theorem from \cite{Wyman/72}, $\Gal(\Q(E[3])/\Q)$ is abelian if and only if the function $p \mapsto |E[3](\F_p)|$ is quasipolynomial.  With one of the classical definitions of heights for elliptic curves (defined for short Weierstrass equations as $\mathrm{ht}(E)=\max\{|a|^3,|b|^2\}$), Theorem~1 from \cite{Duke} implies that the density of isomorphism classes of rational elliptic curves for which $\Gal(\Q(E[3])/\Q)$ is abelian is $0$. In particular, this tells us that examples of groups $\GP_{E,P}$ as in \cref{cor:POFS} for which $p\mapsto|\Aut(\GP_{E,P}(\F_p))|$ is quasipolynomial are very rare.
\end{remark}

Theorems~\ref{mainthm:sat} and~\ref{mainthm:aut} and \cref{cor:POFS} concern
the groups from \cref{sec:G_EP}, which might not appear to be
representative of the class of $E$-groups. However, up to equivalence, which we
define in \cref{sec:determinantal}, these groups seem to occur with probability
$1/2$ in the precise sense explained in  \cref{sec:implementation}.

One of the main tools we use to extract the core geometric properties of the
groups $\GP_{E, P}(\ff)$ is the adjoint algebra, which has recently been
used to understand the structure of $p$-groups~\cite{BOW/19,
BrooksbankWilson/12}. We define  adjoint algebras in \cref{sec:adj-algebras}, but
remark that their isomorphism types are polynomial-time computable isomorphism
invariants of $p$-groups since they are constructed by solving 
linear systems.  With
the notation from \cref{sec:adj-algebras}, \cref{mainthm:adj} follows from
combining the more general 
\cref{thm:reducible-adjoints} with \cref{prop:JP-adj}.  The relevance of $\cor{O}$ being a flex point is
explained in \cref{rmk:flex}. In \cref{mainthm:adj},  we use the fact that a smooth cubic in the projective plane with a marked rational point has the structure of an elliptic curve. 

\begin{thmabc}\label{mainthm:adj} 
Let $\alg$ be a field with $6\alg = \alg$.
    Let, moreover, $\tuB\in\Mat_6(\alg[y_1,y_2,y_3]_1)$ be skew-symmetric with
    $\Pf(\tuB)=0$ defining a smooth cubic $E$ in $\P^2_{\alg}$ with a flex point $\cor{O}\in
    E(\alg)$, and set $G = \GP_{\tuB}(\alg)$. Then
    the following hold:
    \begin{enumerate}[label=$(\arabic*)$] 
        \item There is $P\in E(\alg)\setminus E[2](\alg)$ with $G\cong \GP_{E,P}(\alg)$ if and only if $\Adj(\tuB)\cong \mathbf{X}_1(\alg)$.
        \item There is $P\in E[2](\alg)\setminus\{\cor{O}\}$ with $G\cong \GP_{E,P}(\alg)$ if and only if $\Adj(\tuB)\cong \mathbf{S}_2(\alg)$.
    \end{enumerate}
\end{thmabc}

Our next main theorem allows us to constructively recognize elliptic $p$-groups
and to decide whether two such groups are isomorphic. Deciding whether two
groups of order $n$ are isomorphic uses, in the worst case, $n^{O(\log n)}$
operations \cite{Miller/78}, which is just a brute-force search, and the same
timing prevails for constructing generators of the automorphism group. General
purpose algorithms, like \cite{CH:isomorphism} and
\cite{ELGO:pgrp-automorphism}, use induction by constructing known
characteristic subgroups, which are subgroups fixed by the automorphism group.
Even with recent tools to uncover more characteristic
subgroups~\cite{BOW/19,M17,M21,Wilson/13}, elliptic $p$-groups evade capture.
Prior to this work, it seems that the best general purpose
algorithm~\cite{IvanyosQiao/19}, together with the reduction from
\cref{prop:reduction}, would construct generators for their automorphism group
using $O(|G|^{8/9}\log |G|)$ operations (without \cref{prop:reduction} the
number of operations is $|G|^{O(\log|G|)}$). We significantly improve upon this
timing.

\begin{figure}[h]
    \centering
    \includegraphics[scale=0.9]{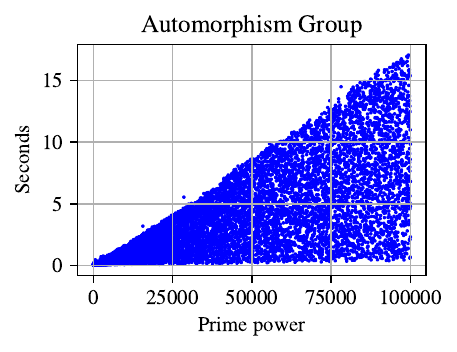}
    \caption{The runtimes of a \textsf{Magma} implementation of
    \cref{mainthm:iso} on the prime powers $p^e$, for $p\notin \{2,3\}$, up to
    $10^5$.}
    \label{fig:data}
\end{figure}

\begin{thmabc}\label{mainthm:iso}
    There are algorithms that, given groups $G_1$ and $G_2$ of order $p^{9m}$,
    with $p\geq 5$, 
    \begin{enumerate}[label=$(\roman*)$]
        \item\label{thm-part:recog1} decide if, for each $i\in\{1,2\}$ and field $\ff$ of cardinality $p^m$,  there exist 
        \begin{itemize}
        \item an elliptic curve $E_i$ given by a short Weierstrass equation with $\ff$-coefficients, 
        \item a point $P_i\in E_i(\ff)\setminus\{(0:1:0)\}$, 
        \end{itemize}
       such that  $G_i$ is isomorphic to $\GP_{E_i,P_i}(\ff)$,
  and if so
        \item\label{thm-part:iso} return the possibly empty coset of
        isomorphisms $G_1\to G_2$.
    \end{enumerate}
 The algorithm for~\ref{thm-part:recog1} is of Las Vegas type and uses $O(m^7 +
    m\log p)$ field operations. The algorithm for~\ref{thm-part:iso} uses $O(p^m)$ field operations.
\end{thmabc}

 Some of the algorithms we use are of \emph{Las Vegas} type. These
are randomized algorithms that only return correct answers and for some
probability, determined by a user-prescribed upper bound, terminate without an
answer.

In \cref{mainthm:iso}\ref{thm-part:iso}, if the groups are isomorphic, the
algorithm returns an isomorphism $G_1\to G_2$ together with a generating set for
$\Aut(G_1)$. In order to demonstrate the efficacy of \cref{mainthm:iso}, we have
constructed generators for the automorphism groups of several instances of
$\GP_{E, P}(\ff)$, built uniformly at random as explained in
\cref{sec:implementation}.

\begin{remark}
As the careful reader has noticed, in our main theorems $p$ is at least $5$.  The prime $2$ is excluded to begin with because there does not exist a group of class $2$ and exponent $2$. Moreover, if $p\in\{2,3\}$, elliptic curves need not admit a short Weierstrass form and, in this case,  their $3\times 3$ determinantal representations might not be equivalent to any $J_{E,P}$; see \cref{sec:G_EP} and \cref{sec:beauville}.
\end{remark}

\smallskip
\smallskip
\noindent
{\bf Acknowledgements.} 
We wish to thank Daniele Agostini for his precious help in connection to the proof of \cref{mainthm:sat}; Mima especially thanks Daniele for always joyfully welcoming her questions. We also thank Fulvio Gesmundo, Carlo Pagano, and Rosa Winter for pointing out \cite{Ng}, \cite{Duke}, and \cite{BanPal/12} to us. We thank Andrea Bandini for some clarifications around \cite{BanPal/12}. We, moreover, thank Francesco Galuppi, Tobias Rossmann, Christopher Voll for their precious comments on an early version of this manuscript. {{We thank Lars B\"ugemannskemper for pointing out a typo involving the order of the central automorphisms.} 
We are grateful to the anonymous referees for their careful study of our paper and for their kind and detailed reports, which helped us fix some imprecisions and led to an improvement of the paper. 

This project was initiated during a Research in Pairs visit at MFO in March 2022
and was partially supported by the Daimler and Benz Foundation. The first author
was supported by the DFG-Graduiertenkolleg ``Mathematical Complexity Reduction''
and the DFG grant VO 1248/4-1 (project number 373111162). The second author was
supported by the DFG -- Project-ID 286237555 -- TRR 195 and by the Italian program ``Rita Levi Montalcini'', edition 2020.

\section{Tensors and unipotent group schemes} 

Fix finite-dimensional
$\alg$-vector spaces $U$, $W$, $V$, $V'$, and $T$. By a \emph{$3$-tensor} (or simply \emph{tensor}, throughout),
we mean a $\alg$-bilinear map $t:V\times V'\rightarrow T$, that is, for all
$u,v\in V$, all $u',v'\in V'$, and all $\lambda,\mu\in \alg$, the following
holds
\[
t(u + \lambda v, u'+\mu v') = t(u, u') + \lambda t(v,u')+\mu t(u,v')+\lambda\mu t(v,v').
\]
A tensor $t:V\times V\rightarrow T$ is \emph{alternating} if, for all $v\in V$,
one has $t(v,v)=0$.

\begin{defn} 
Let $t:V\times V\rightarrow T$ be a tensor. 
    A subspace $U\leq V$ is \emph{totally
    isotropic} with respect to $t$ if, for every $u,u'\in U$, one has $t(u,u')=0$, in other words if the restriction of $t$ to $U\times U$ is the zero map. 
\end{defn} 

Note that, if $t:V\times V\rightarrow T$ is alternating, every line in $V$ is totally isotropic.

\begin{defn} 
    An alternating tensor $t : V\times V\rightarrow T$ is \emph{isotropically
    decomposable} if there exist totally isotropic subspaces $U, W\leq V$ such
    that $V=U\oplus W$.
\end{defn}

By choosing bases for $U$, $W$, and $T$, we may write $t:U\times W\rightarrow T$
as a matrix of linear forms or as a system of forms; that is, a sequence of
matrices over $\alg$. Let $\{e_1,\dots, e_m\}$, $\{f_1,\dots, f_n\}$, and
$\{g_1,\dots, g_d\}$ be bases for $U$, $W$, and $T$, respectively. For $i\in
[m]$, $j\in [n]$, and $k\in [d]$, define $b_{ij}^{(k)}\in \alg$ such that 
\begin{align*} 
    t(e_i, f_j) &= \sum_{k=1}^d b_{ij}^{(k)} g_k.
\end{align*} 
The matrix of linear forms $\tuB =(b_{ij}) \in \Mat_{m\times n}(\alg[\bm{y}]_1)$
corresponding to $t$ is given by 
\begin{align*} 
    b_{ij} &= \sum_{k=1}^d b_{ij}^{(k)} y_k. 
\end{align*} 
In the sequel we will always assume that a tensor $t:U\times W\rightarrow T$ is given together with a choice of bases for $U,W,T$. If $t:V\times V\rightarrow T$ is an alternating tensor, we take the same basis on the first and the second copy of $V$ and, if $V$ is given by an isotropic decomposition $U\oplus W$, we assume that the basis on $V$ is the composite of a basis of $U$ and a basis of $W$.

\begin{remark}
    Since every $3$-tensor $t$ is, with respect to a choice of bases, given by a matrix $\tuB$ of linear forms,  throughout the paper everything that is defined for tensors will also apply to matrices of linear forms.
\end{remark}

If $t : V\times V \to T$ is an alternating tensor, just as matrices of linear
forms associated with $t$ vary as the bases for $V$ and $T$ vary, so does the
corresponding Pfaffian. We will say that $\Pf(t)\in \alg[{\bm y}]$ is a Pfaffian
for $t$ if there exists some choice of bases for $V$ and $T$ whose associated
matrix of linear forms $\tuB$ satisfies $\Pf(t)=\Pf(\tuB)$. 

\begin{defn} 
    A tensor $t:U\times W\rightarrow T$ is \emph{nondegenerate} if the following
    are satisfied: 
    \begin{enumerate}[label=(\emph{\roman*})] 
        \item $t(u, W)=0$ implies $u=0$, and 
        \item $t(U, w)=0$ implies $w=0$.
    \end{enumerate}
    We say that $t$ is \emph{full} if $t(U, W) = T$. A tensor $t$ is \emph{fully
    nondegenerate} if $t$ is both nondegenerate and full. 
\end{defn}

\subsection{Maps between $3$-tensors}

In this section we present various types of maps between tensors that we will exploit in later sections for the determination of isomorphisms between groups and Lie algebras. Many of these definitions can be found in~\cite{BOW/19}.

\begin{defn}\label{def:equivalence-tensor}
    Two $\alg$-tensors $s : U\times W \rightarrow T$ and $t:U'\times
    W'\rightarrow T'$ are \emph{equivalent} if $T=T'$ and there exist
    $\alg$-linear isomorphisms $\alpha:U\rightarrow U'$ and $\beta :
    W\rightarrow W'$ such that for all $u\in U$ and all $w\in W$, one has
    $t(\alpha(u), \beta(w)) = s(u, w)$; equivalently, the following diagram
    commutes
    \[
        \xymatrix{
        U \ar@<-2.7ex>[d]_{\alpha} \ar@<2ex>[d]^{\beta} \times W \ar[rr]^s & & T\ar[d]^{\id}\\
        U' \times W' \ar[rr]^{t} && T'
        }
    \]
    In the case $U=W$, $U'=W'$, and $\alpha=\beta$, we say the two tensors are \emph{isometric}.
\end{defn}

Relevant to our context is a different form of equivalence from
\cref{def:equivalence-tensor}, namely, \emph{pseudo-isometry}. Isomorphisms of groups
induce this weaker equivalence, which is considered in more detail in the next section. 

\begin{defn}\label{defn:pseudo-isometry}
    Let $L\subset \alg$ be a subfield and let  $s : V\times V \rightarrow T$ and
    $t:V'\times V'\rightarrow T'$ be $\alg$-tensors.  An
    \emph{$L$-pseudo-isometry} from $s$ to $t$ is a pair $(\alpha, \beta)$ such
    that $\alpha:V\rightarrow V'$ and $\beta : T\rightarrow T'$ are $L$-linear
    isomorphisms and for all $u,v \in V$,  the equality $t(\alpha(u), \alpha(v))
    = \beta(s(u, v))$ holds; equivalently the following diagram commutes:
    \[
        \xymatrix{
        V \ar@<-2.7ex>[d]_{\alpha} \ar@<2ex>[d]^{\alpha} \times V \ar[rr]^s & & T\ar[d]^{\beta}\\
        V' \times V' \ar[rr]^{t} && T'
        }
    \]
    The set of $L$-pseudo-isometries from $s$ to $t$ is denoted $\pseudo_L(s,t)$
    and, if $\pseudo_L(s,t)$ is non-empty, $s$ and $t$ are called
    $L$-\emph{pseudo-isometric}. If $s=t$, then the group of
    $L$-pseudo-isometries of $t$ is denoted by $\pseudo_L(t)$ instead of
    $\pseudo_L(t,t)$.
\end{defn}

\begin{remark}
    We  emphasize that $\pseudo_L(s,t)$ has more structure than an arbitrary
    set. Indeed, identifying $V$ with $V'$ and $T$ with $T'$, we can view $\pseudo_L(s, t)$ both as a left $\pseudo_L(t)$-coset and a right
    $\pseudo_L(s)$-coset of $\GL_L(V)\times \GL_L(T)$. That is, if
    $(\alpha,\beta) \in \pseudo_L(s, t)$,  while $(\gamma,\delta)\in \pseudo_L(t)$ and
    $(\epsilon, \zeta) \in \pseudo_L(s)$, then the following holds: 
    \begin{align*} 
        (\gamma \alpha, \delta \beta), (\alpha \epsilon, \beta \zeta) \in \pseudo_L(s, t). 
    \end{align*} 
    Because $\pseudo_L(t)\cdot (\alpha, \beta) = (\alpha,
    \beta)\cdot\pseudo_L(s) = \pseudo_L(s, t)$, we will just refer to
    $\pseudo_L(s, t)$ as a coset.
\end{remark}

\begin{remark}
    If $G$ is a $p$-group with central elementary abelian derived subgroup, then its commutator determines a fully
    nondegenerate $\F_p$-tensor $t_G : G/\ZG(G)\times G/\ZG(G) \rightarrow
    G'$ given by $(g\ZG(G), h\ZG(G))\mapsto [g,h]$. If $G$ and $H$ are two such
    $p$-groups whose tensors $t_G$ and $t_H$ are
    $\F_p$-pseudo-isometric, then $G$ and $H$ are \emph{isoclinic}, cf.\
    \cite{Hall}. If, in addition, both $G$ and $H$ have exponent $p$, then $G$
    and $H$ are isomorphic.  
\end{remark}

\begin{defn}
    Let $L\subset \alg$ be a subfield and let  $s : U\times W \rightarrow T$ and
    $t:U'\times W'\rightarrow T'$ be $\alg$-tensors.  An
    \emph{$L$-isotopism} from $s$ to $t$ is a triple $(\alpha, \beta,\gamma)$ such
    that $\alpha:U\rightarrow U'$, $\beta : W\rightarrow W'$, and $\gamma:T\rightarrow T'$ are $L$-linear
    isomorphisms and for all $(u,w) \in U\times W$,  the equality $t(\alpha(u), \beta(w))
    = \gamma(s(u, w))$ holds; equivalently the following diagram commutes:
    \[
        \xymatrix{
        U \ar@<-2.7ex>[d]_{\alpha} \ar@<2ex>[d]^{\beta} \times W \ar[rr]^s & & T\ar[d]^{\gamma}\\
        U' \times W' \ar[rr]^{t} && T'
        }
    \]
If $s=t$, then the triple $(\alpha,\beta,\gamma)$ is called an \emph{$L$-autotopism} of $t$ and the group of $L$-autotopisms of $t$ is denoted by $\mathrm{Auto}_L(t)$.
\end{defn}

\begin{remark}\label{rmk:autotopism-pseudo}
Assume $\tuM\in\Mat_n(\ff[y_1,\dots, y_d]_1)$ realizes the tensor $\tilde{t}:U \times W\rightarrow T$. Then the autotopism group of $\tuM$ can be described as
\begin{align*} 
    \mathrm{Auto}_{\ff}(\tuM) &= \left\{ (X, Y, Z) \in \GL_n(\ff)\times \GL_n(\ff)\times \GL_d(\ff) ~\middle|~ X^{\tp}\tuM(\bm{y})  Y = \tuM(Z\bm{y}) \right\} . 
\end{align*} 
In particular,  if $U\oplus W=V$ is an isotropic decomposition for the tensor $t:V\times V\rightarrow T$ inducing $\tilde{t}$ and with associated matrix $\tuB$, then every triple $(X,Y,Z)\in\mathrm{Auto}_{\ff}(\tilde{t})$ yields an element of $\pseudo_{\ff}(t)$ via the following:
\begin{align*}
    \begin{pmatrix} X^{\tp} & 0 \\ 0 & Y^{\tp} \end{pmatrix} \tuB\begin{pmatrix} X & 0 \\ 0 & Y \end{pmatrix} & =     \begin{pmatrix} X^{\tp} & 0 \\ 0 & Y^{\tp} \end{pmatrix} \begin{pmatrix}
   0 & \tuM \\ -\tuM^{\tp} & 0 
    \end{pmatrix} \begin{pmatrix} X & 0 \\ 0 & Y \end{pmatrix} \\
    & = \begin{pmatrix}
    0 & X^{\tp}\tuM(\bm{y})Y \\ - Y^{\tp}\tuM(\bm{y})^{\tp}X & 0
    \end{pmatrix} = \begin{pmatrix}
    0 & \tuM(Z\bm{y}) \\ -\tuM(Z\bm{y})^{\tp} & 0
    \end{pmatrix}= \tuB(Z\bm{y}).
\end{align*} 
\end{remark}

\subsection{Algebras associated to $3$-tensors}\label{sec:adj-algebras}
Throughout this subsection, we use tensors rather than matrices of linear forms,
but the content applies to both. These algebras are also found in~\cite{BOW/19}.

Let $t:U\times W\rightarrow T$ be a
$\alg$-tensor. The \emph{centroid} of $t$ is the $\alg$-subalgebra of
$\mathcal{E} = \End_{\alg}(U) \times \End_{\alg}(W) \times \End_{\alg}(T)$
defined by
\begin{align}\label{eq:centroid}
    \Cent(t) &= \left\{(\alpha, \beta, \gamma) \in \mathcal{E} ~\middle|~ \begin{array}{c} \forall u\in U, \; \forall w\in W, \\ t(\alpha(u), w) = t(u, \beta(w)) = \gamma(t(u,w)) \end{array} \right\} .
\end{align}
Let $c=(\alpha,\beta,\gamma)\in\Cent(t)$, and suppose $\Cent(t)$ acts on $U$, $W$, and $T$
by respectively applying $\alpha$, $\beta$, and $\gamma$. Then for each pair $(u,w)\in U\times W$, we have 
\begin{align*}
    c\cdot t(u,w) &= t(c\cdot u, w) = t(u, c\cdot w). 
\end{align*}
Thus, $t$ is $\Cent(t)$-bilinear. Additionally, the centroid satisfies the following universal property. If $t$ is also
$A$-bilinear for some $\alg$-algebra $A$, then there is a unique ring homomorphism $A\rightarrow\Cent(t)$ such that the action of $A$ on $U$, $W$, and $T$ is that of $\Cent(t)$; cf.~\cite[Lem.~6.8$(ii)$]{Wilson/12}.

\begin{lem} \label{lem:comm-cent}
    If $t$ is fully-nondegenerate, then $\Cent(t)$ is commutative. 
\end{lem}

\begin{proof}
   Take $(\alpha,\beta,\gamma),(\alpha',\beta',\gamma')\in \Cent(t)$; then
    for all $u\in U$ and all $w\in W$, 
    \begin{align*} 
        t(\alpha\alpha'(u), w) &= t(\alpha'(u), \beta(w)) 
        = \gamma'(t(u, \beta(w))) = \gamma'(t(\alpha(u), w))
        = t(\alpha'\alpha(u), w). 
    \end{align*} 
    This implies that $t((\alpha\alpha' - \alpha'\alpha)(u), w) = 0$, and since
    $t$ is fully-nondegenerate, it follows that $\alpha\alpha' = \alpha'\alpha$.
    Similar arguments hold for the other coordinates as well. 
\end{proof}

For a ring $R$, we denote by $R^{\mathrm{op}}$ the \emph{opposite ring} of $R$,
with opposite product given by $x\cdot_{\mathrm{op}}y=yx$ for all $x,y\in R$.
The \emph{adjoint algebra} of $t$ is the $\alg$-algebra
\begin{align} \label{eq:adjoint}
    \Adj(t) &= \left\{(\alpha, \beta) \in \End(U)\times \End(W)^{\mathrm{op}} ~\middle|~ \begin{array}{c} \forall u\in U, \; \forall w\in W, \\ t(\alpha(u), w) = t(u, \beta(w)) \end{array} \right\}. 
\end{align} 
If $t : V\times V\rightarrow T$ is alternating, then the natural
anti-isomorphism $*:\Adj(t)\rightarrow \Adj(t)^{\mathrm{op}}$, given by $(\alpha,
\beta) \mapsto (\beta, \alpha)$, makes $\Adj(t)$ a \emph{$*$-algebra},
cf.\ Section~\ref{subsec:*-rings}. To see this, let $(\alpha,\beta)\in\Adj(t)$.
Then for all $u,v\in V$, one has
\begin{align*} 
    t(\beta(u), v) = -t(v, \beta(u)) = -t(\alpha(v), u) = t(u, \alpha(v)),
\end{align*} 
from which it follows that $(\beta,\alpha)\in\Adj(t)^{\mathrm{op}}$.

We define a module version of the adjoint algebra. For this, let $t : U\times
W\rightarrow T$ and $s : U'\times W'\rightarrow T'$ be $\alg$-tensors. Identifying $T$ with $T'$, the
\emph{adjoint module} of $s$ and $t$ is 
\begin{align}\label{eq:adjoint-space}
    \Adj(s, t) &= \left\{(\alpha,\beta)\in\Hom(U',U)\times \Hom(W,W') ~\middle|~ \begin{array}{c} \forall u'\in U', \; \forall w\in W, \\ t(\alpha(u'), w) = s(u', \beta(w)) \end{array} \right\} .
\end{align} 
We remark that $\Adj(s, t)$ is a left $\Adj(t)$-module and a right
$\Adj(s)$-module via defining
\[
    (\gamma, \delta) \cdot (\alpha, \beta) = (\gamma\alpha, \beta\delta)\ \ \textup{ and } \ \  (\alpha, \beta) \cdot (\epsilon, \zeta) = (\alpha\epsilon, \zeta\beta),
\]
for all $(\alpha,\beta)\in\Adj(s, t)$, all $(\gamma,\delta)\in\Adj(t)$, and all
$(\epsilon, \zeta)\in\Adj(s)$. 

The key computational advantage to the vector spaces in \eqref{eq:centroid},
\eqref{eq:adjoint}, and \eqref{eq:adjoint-space} is that they are simple to
compute, since they are given by a system of linear equations, and carry useful
structural information. Although they are computed by solving a linear system,
this is the computational bottleneck for
\cref{mainthm:iso}\ref{thm-part:recog1}.

\subsection{Structure of Artinian $*$-rings}\label{subsec:*-rings}

We will need specific structural information about $*$-rings. We follow the
treatment given in \cite[Sec.~2]{BrooksbankWilson/12} and, only in \cref{subsec:*-rings}, use right action in accordance to the existing literature on $*$-algebras.  

A \emph{$*$-ring} is a pair $(A,*)$ where $A$ is an Artinian ring and
$*:A\rightarrow A^{\mathrm{op}}$ is a ring homomorphism such that, for all $a\in A$, the equality $(a^*)^*=a$ holds. A \emph{$*$-homomorphism} $\phi : (A,*_A) \rightarrow (B,*_B)$
is a ring homomorphism $A\rightarrow B$ such that $\phi(a^{*_A}) =
\phi(a)^{*_B}$. A \emph{$*$-ideal} is an ideal $I$ of $A$ such that $I^*=I$.
A $*$-ring $A$ is \emph{simple} if its only $*$-ideals are $0$ and $A$.
Moreover, from~\cite[Th.~2.1]{BrooksbankWilson/12}, the Jacobson radical of
a $*$-ring is a $*$-ideal, and the quotient by the Jacobson radical $R$ decomposes
into a direct sum of simple $*$-rings. If $A$ is a $*$-ring with Jacobson radical $R$ such that $S= A/R$, then we write $A\cong R\rtimes S$. 

For finite fields of odd characteristic, we can completely describe the simple
$*$-rings. Using the notation from \cite{BrooksbankWilson/12}, we have the
following classification.

\begin{thm}[{\cite[Th.~2.2]{BrooksbankWilson/12}}]\label{thm:simple-star-algebras}
    Let $(A, *)$ be a simple $*$-algebra over $\ff$. Then there is a positive
    integer $n$ such that $(A,*)$  is $*$-isomorphic to one of the following:
    \begin{enumerate}[label=$(\roman*)$]
        \item $\mathbf{O}_n^{\varepsilon}(\ff) = (\Mat_n(\ff), X\mapsto DX^{\tp} D^{-1})$, where $D$ is one of the diagonal matrices in $\{I_n, I_{n-1}\oplus \omega\}$, for a non-square $\omega\in\ff$, and $\varepsilon\in\{+,-,\circ\}$ according to whether the form $(u,v)\mapsto u^{\tp}Dv$ induces an orthogonal geometry of type $\varepsilon$.
        \item $\mathbf{U}_n(L) = (\Mat_n(L), X\mapsto \overline{X}^{\tp})$, where $\alpha\mapsto \overline{\alpha}$ is the nontrivial Galois automorphism of the degree two extension $L/\ff$.
        \item $\mathbf{S}_{2n}(\ff) = (\Mat_{2n}(\ff), X \mapsto JX^{\tp}J^{-1})$, where $J$ is the $n$-fold direct sum of $\left(\begin{smallmatrix} 0 & 1 \\ -1 & 0 \end{smallmatrix}\right)$. 
        \item $\mathbf{X}_n(\ff) = (\Mat_n(\ff)\oplus \Mat_n(\ff), (X, Y) \mapsto (Y^{\tp}, X^{\tp}))$. 
    \end{enumerate} 
\end{thm}

We stress that in this paper we are mostly concerned with $(iii)$ and $(iv)$ from the last classification.  We  mention $(i)$ primarily for \cref{thm:reducible-adjoints} for the values $n=1$ and $\varepsilon=\circ$, in which case $\mathbf{O}_n^{\varepsilon}(\alg)=\mathbf{O}_1^{\varepsilon}(\alg)$ is a field isomorphic to $\alg$.

\subsection{Unipotent group schemes and nilpotent Lie algebras from $3$-tensors}\label{sec:baer3}

Let $t:V\times V\rightarrow T$ be an alternating $\alg$-tensor and assume that
$\alg=2\alg$. Then the set $V\oplus T$, with multiplication given by 
\begin{align*} 
    (v, w) \cdot (v', w') &= (v + v', w + w' + \dfrac{1}{2}t(v, v')), 
\end{align*} 
is a nilpotent group (of class at most $2$),
called the \emph{Baer group associated with $t$}.
The commutator map of $\GP_t(\alg)=(V\oplus T, \cdot)$ is given by 
\begin{align*} 
 ((v, w), (v', w')) \longmapsto  \left[(v, w), (v', w')\right] &= (0, t(v, v')).
\end{align*} 
Interpreting the last commutator map as a Lie bracket, the Lie algebra $\LA_{t}(\alg)$ associated to $t$ is precisely $V\oplus T$ together with this Lie product.  We note that this construction is the same as in~\cite{stanojkovski2019hessian},
where $t$ is assumed to be isotropically decomposable (and the associated half-matrix $\tuB$ to be symmetric). 

The (Baer) group scheme $\GP_t$ associated to $t$ is determined in the following
way -- without requiring $\alg=2\alg$; see~\cite[Sec.~2.4]{RV/2019}. Fix bases
$(v_1,\dots, v_n)$ and $(w_1,\dots, w_d)$ of $V$ and $T$ respectively. Let $L$
be an associative commutative unital $\alg$-algebra, and identify $\GP_t(L)$
with the set $(V\otimes_\alg L)\oplus (T\otimes_\alg L)$.  For
$\ell\in L$, we abbreviate $v_i\otimes \ell$ and $w_j\otimes \ell$ to $\ell v_i$
and $\ell w_j$ in $V\otimes_\alg L$ and $T\otimes_\alg L$ respectively.

We define the multiplication $\bullet$ on $\GP_t(L)$ as follows.  
For $v=a_1v_1+ \cdots + a_nv_n$ and $v'=b_1v_1+\cdots + b_nv_n$ with $a_1,\dots,
a_n,b_1,\dots, b_n\in L$, we set 
\[ 
    v\bullet v' = v + v' - \sum_{1\leq i < j \leq n}a_jb_i\cdot t(v_i, v_j). 
\]
 Additionally, for all $x\in
\GP_t(L)$ and $w\in T\otimes_KL$, define $x\bullet w=x+w=w\bullet x$. 
This determines the group structure and
is independent of the chosen bases: shall the tensors be given in terms of a
matrix $\tuB$ of linear forms, we will write $\GP_{\tuB}(L)$ for the resulting
group. The last construction defines a representable functor from the category
of $\alg$-algebras (associative, commutative, and unital) to groups; namely, via
$L\mapsto \left( (V\otimes_\alg L)\oplus (T\otimes_\alg L), \bullet\right)$.
Concretely, if $\alg=2\alg$, then $\GP_t(\alg)$ is isomorphic to the Baer group
associated with $t$  since the two groups yield pseudo-isometric commutator tensors. One approaches Lie algebras in a similar fashion.

Our main source of unipotent group schemes comes from linear determinantal
representations of elliptic curves, but this change of perspective from
determinantal representations to nilpotent groups comes with a few subtleties.
For example, the specific Pfaffian hypersurface associated to the linear
Pfaffian representation is an invariant in the equivalence class as given in \cref{def:equivalence}. However for
nilpotent groups of class 2, the Pfaffian hypersurface is not an invariant of
the group -- the next theorem, essentially due to Baer~\cite{Baer/38},
illustrates this point.

\begin{thm}\label{thm:Baer-correspondence}
    Let $s: V\times V \rightarrow T$ and $t:V'\times V'\rightarrow T'$ be
    alternating, full $\alg$-tensors, where $\charac(\alg) = p >2$. Then the
    following hold.
    \begin{enumerate}[label=$(\arabic*)$] 
        \item The $\alg$-Lie algebras $\LA_s(\alg)$ and  $\LA_t(\alg)$ are
        isomorphic if and only if $s$ and $t$ are $\alg$-pseudo-isometric.
        \item The groups $\GP_s(\alg)$ and  $\GP_t(\alg)$ are isomorphic if and only
        if $s$ and $t$ are $\F_p$-pseudo-isometric. 
    \end{enumerate}
\end{thm} 

\begin{proof}
    We prove the claim for groups as the Lie algebras statement is similar.
    Since $s$ and $t$ are full, the commutator subgroups of $\GP_s(\alg)$ and
    $\GP_t(\alg)$ are equal to $0\oplus T$ and $0\oplus T'$, respectively. An
    isomorphism from $\GP_s(\alg)$ to $\GP_t(\alg)$ induces an isomorphism of their commutator subgroups and their abelianizations. These yield $\F_p$-linear isomorphisms 
    $V\rightarrow V'$ and $T\rightarrow T'$, so $s$ and $t$ are
    $\F_p$-pseudo-isometric. 

    Conversely if $\alpha: V\rightarrow V'$ and $\beta : T\rightarrow T'$ are
    $\F_p$-linear isomorphisms with $(\alpha,\beta)\in\pseudo_{\F_p}(s,t)$, then for all $v,v'\in V$ and all
    $w,w'\in T$,  the following holds:
    \begin{align*} 
        (\alpha(v), \beta(w)) \cdot (\alpha(v'), \beta(w')) &= (\alpha(v+v'), \beta(w+w') + \dfrac{1}{2}t(\alpha(v), \alpha(v'))) \\ 
        &= (\alpha(v+v'), \beta(w+w'+\dfrac{1}{2}s(v, v'))) . 
    \end{align*} 
 In other words, $\diag(\alpha,\beta)$ yields an isomorphism of groups. 
\end{proof}

For an $\ff$-vector space $V$ with a fixed basis $\mathcal{B}$,  we extend the action of $\Gal(\ff/\F_p)$ from $\ff$ to $V$ in the following way:
\begin{align*}
\Gal(\ff/\F_p) & \longrightarrow \Aut_{\F_p}(V), \quad 
\sigma  \longmapsto \left(u=\sum_{b\in \mathcal{B}} c_b b   \mapsto \sigma(u) = \sum_{b\in\mathcal{B}} \sigma(c_b) b\right).
\end{align*}

\begin{defn}
Let $t:V\times V \to T$ be an $\ff$-tensor and let $\sigma \in
\Gal(\ff/\F_p)$. For a fixed choice of bases for $V$ and $T$, the tensor
${}^{\sigma}t : V \times V \to T$ is defined by
\begin{align*} 
    {}^{\sigma}t(u, v) &= \sigma t(\sigma^{-1}(u), \sigma^{-1}(v)) . 
\end{align*} 
On the level of matrices of linear forms, if $t(u,v)=u^{\tp}\tuB v$, then
${}^{\sigma}t(u,v) = u^{\tp}(\sigma(\tuB)) v$, where the action of $\sigma$ on
$\tuB$ is entry-wise, so $\sigma(\tuB) = (\sigma(b_{ij}))$. 
\end{defn}

The action of $\Gal(\ff/\F_p)$ on an $\ff$-tensor $t: V\times V\to T$
depends on the choice of $\ff$-bases for $V$ and $T$, and different choices of
bases may yield different $\ff$-pseudo-isometry classes. This shall not concern
us because this does not happen on the level of
$\F_p$-pseudo-isometries -- our primary focus when working with
$\Gal(\ff/\F_p)$ -- as $\sigma$ is $\F_p$-linear. 

The next theorem has many appearances in different guises~\cite{BMW/17,
stanojkovski2019hessian, Wilson/17}, and it describes the group
$\pseudo_{\F_p}(t)$ as a subgroup of a direct product of
$\ff$-semilinear groups. The \emph{$\ff$-semilinear group} of $V$ is
$\GL_{\ff}(V)\rtimes \Gal(\ff/\F_p)$, and for a fixed basis of $V$, it
acts on $V$ by mapping $v$ to $(X, \sigma) v = X\sigma(v)$. For $(X, \sigma),
(Y, \tau) \in \GL_{\ff}(V)\rtimes \Gal(\ff/\F_p)$, we compute that $(X, \sigma)
\cdot (Y, \tau) = (X\sigma Y\sigma^{-1}, \sigma\tau)$, and we
extend this operation to the larger $(\GL_{\ff}(V)\times \GL_{\ff}(T)) \rtimes
\Gal(\ff/\F_p)$ by setting
\begin{equation}\label{eq:operation}
(\alpha, \beta, \sigma) \cdot (\gamma, \delta,
\tau) = (\alpha \sigma\gamma \sigma^{-1}, \beta \sigma\delta
\sigma^{-1}, \sigma\tau).  
\end{equation} 
Then the \emph{$\ff$-semilinear pseudo-isometry group} of
$t$ is 
\begin{align*} 
    \spseudo_{\ff/\F_p}(t) &= \{(\alpha, \beta, \sigma) \in (\GL_{\ff}(V)\times \GL_{\ff}(T)) \rtimes \Gal(\ff/\F_p) ~|~ (\alpha,\beta)\in \pseudo_{\ff}({}^{\sigma}t, t)\} 
\end{align*} 
with the induced multiplication.
Note that, though the definition of $\spseudo_{\ff/\F_p}(t)$ we gave clearly depends on a choice of bases of $V$ and $T$, its isomorphism type does not.
We will write $\Gal_t(\ff/\F_p) $ for the following:
\begin{equation}\label{Gal_t}
\Gal_t(\ff/\F_p) = \{\sigma \in\Gal(\ff/\F_p) ~|~
    \pseudo_{\ff}({}^{\sigma}t, t)\neq\varnothing\}.
    \end{equation}
By using the operation
    from \eqref{eq:operation}, a calculation shows that
    $\Gal_t(\ff/\F_p)$ is a subgroup of $\Gal(\ff/\F_p)$. 

\begin{thm}\label{prop:reduction}
    Let $t:V\times V\to T$ be an alternating, fully nondegenerate $\ff$-tensor.
    If $\Cent(t)\cong \ff$, then the following hold:
    \begin{enumerate}[label=$(\arabic*)$]
    \item $\Aut(\GP_{t}(\ff)) \cong \Hom_{{ \F_p}}(V,T) \rtimes \spseudo_{\ff/\F_p}(t)$ and 
    \item $|\spseudo_{\ff/\F_p}(t)| = |\pseudo_{\ff}(t)|\cdot|\Gal_t(\ff/\F_p)|$.
    \end{enumerate}
\end{thm}

\begin{proof} 
    Let $\LA = \LA_t(\ff)$ be the Lie algebra associated to $G = \GP_{t}(\ff)$.
    Since $G$ is of class $2$ with exponent $p>2$, the Baer correspondence
    guarantees that $\Aut(G)\cong \Aut_{\F_p}(\LA)$. Since $t$ is full,
    $T = [\LA, \LA]$, so every endomorphism of $\LA = V\oplus T$ maps $T$ into
    $T$. Thus, we have a split exact sequence of groups
    \begin{align*}
        1 &\longrightarrow \Hom_{\F_p}(V, T) \longrightarrow \Aut_{\F_p}(\LA) \longrightarrow \pseudo_{\F_p}(t) \longrightarrow 1, 
    \end{align*} 
    where the penultimate map is given by $\alpha\mapsto (\alpha_V, \alpha_T)$
    with $\alpha_V : V \to V$ given by $v + T \mapsto \alpha(v) + T$ and
    $\alpha_T = \alpha|_T$ is the restriction of $\alpha$ to $T$. 

    The group $\pseudo_{\F_p}(t)$ acts on $\Cent(t)$ via conjugation. If
    $(X,Y,Z)\in\Cent(t)$ and $(\alpha,\beta)\in\pseudo_{\F_p}(t)$, then
    $(\alpha X \alpha^{-1}, \alpha Y \alpha^{-1}, \beta Z \beta^{-1})$ belongs to $\Cent(t)$. This defines a homomorphism $\gamma :
    \pseudo_{\F_p}(t) \to \Aut(\Cent(t))$, where $\Aut(\Cent(t))$
    denotes the automorphism group of $\Cent(t)$ as a unital $\F_p$-algebra.
    Since $\Cent(t)\cong \ff$, the kernel of this map is $\pseudo_{\ff}(t)$.
    The choice of an isomorphism $\phi:\Aut(\Cent(t)) \to
    \Gal(\ff/\F_p)$ yields the following exact sequence of groups:
    \begin{align*}
        1 &\longrightarrow \pseudo_{\ff}(t) \longrightarrow \pseudo_{\F_p}(t) \longrightarrow \Gal(\ff/\F_p). 
    \end{align*} 
    We show that $\im(\phi\circ\gamma) = \Gal_t(\ff/\F_p)$. Given
    $\sigma\in\Gal_t(\ff/\F_p)$, we have that $\sigma\in \im(\phi\circ\gamma)$ if and only if there exists
    $(X,Y)\in\GL_{\ff}(V)\times\GL_{\ff}(T)$ such that $(\alpha,\beta)=(X\sigma,
    Y\sigma)\in\pseudo_{\F_p}(t)$. This condition is equivalent to the equation   
    \begin{align}\label{eqn:Galois-action-t}
        t(X\sigma(u),X\sigma(v)) &  = t(\alpha(u),\alpha(v)) = \beta t(u,v) =Y\sigma t(u,v) = Y({}^\sigma t(\sigma(u),\sigma(v))).
    \end{align} 
    being verified for all $u,v\in V$. However,
    since $\sigma$ induces an automorphism of $V$, the outer equality of \eqref{eqn:Galois-action-t}
    is equivalent to having $(X,Y)\in\pseudo_{\ff}({}^\sigma t,t)$, meaning that $\sigma\in \Gal_t(\ff/\F_p)$.
    As a result, the following is a short exact sequence: 
    \begin{align*}
        1 &\longrightarrow \pseudo_{\ff}(t) \longrightarrow \pseudo_{\F_p}(t) \longrightarrow \Gal_t(\ff/\F_p) \longrightarrow 1. 
    \end{align*} 
    In particular, the size of $\pseudo_{\F_p}(t)$ is equal to
    $|\pseudo_{\ff}(t)|\cdot|\Gal_t(\ff/\F_p)|$. 
    
    We conclude by showing that $\pseudo_{\F_p}(t)$ and
    $\spseudo_{\ff/\F_p}(t)$ are isomorphic. For this, note that, if
    $(\alpha, \beta)\in\pseudo_{\F_p}(t)$ is such that
    $\phi\gamma((\alpha, \beta)) = \sigma$, then $(\alpha\sigma^{-1},
    \beta\sigma^{-1})$ belongs to $\pseudo_{\ff}({}^{\sigma}t, t)$. Therefore, the map
    \[ 
        \spseudo_{\ff/\F_p}(t) \longrightarrow \pseudo_{\F_p}(t) \quad\text{given by}\quad (\alpha,\beta, \sigma) \longmapsto (\alpha\sigma, \beta\sigma)
    \]
    defines an isomorphism. Putting everything together, the theorem follows.
\end{proof}

Following \cite[Sec.~4]{GS/84}, we say two
homogeneous polynomials $f,g\in \ff[y_1,\dots,y_d]$ are \emph{projectively
semi-equivalent} if there exist
$M=(m_{ij})\in\GL_d(\ff)$, $\sigma\in \Gal(\ff/\F_p)$, and $\lambda\in
\ff^\times$ such that
\begin{align}\label{def:proj-semi-equiv}
    (\sigma(f))(m_{11}y_1 + \cdots + m_{1d}y_d, \dots, m_{d1}y_1 + \cdots + m_{dd}y_d) = \lambda g(\bm{y}), 
\end{align} 
where $(\sigma(f))(\bm{y})$ is the polynomial obtained by applying $\sigma$ to
the coefficients of $f$. If $\sigma=1$, then we just say that $f$ and $g$ are \emph{projectively
equivalent}. Projective equivalence yields an isomorphism of the varieties of
$f$ and $g$, but the converse is not true in general. Moreover projective
semi-equivalence need not induce an isomorphism of varieties of $f$ and $g$ but
of $\sigma(f)$ and $g$ instead. The following corollary collects some direct
geometric implications of \cref{thm:Baer-correspondence}, some more will be presented in \cref{sec:tenspfaff}.

\begin{coro}
    Assume that $s: V\times V\to T$ and $t: V'\times V'\to T'$ are
    fully-nondegenerate, that $\Cent(t)\cong\Cent(s)\cong \ff$, and that
    $\GP_s(\ff)\cong \GP_t(\ff)$. Let $\Pf(s)$ and $\Pf(t)$ be Pfaffians over
    $\ff$ of $s$ and $t$, respectively. Then the following hold.
    \begin{enumerate}[label=$(\arabic*)$] 
        \item The polynomials $\Pf(s)$ and $\Pf(t)$ are projectively semi-equivalent.
        \item The $\ff$-points of the singular loci of $\Pf(s)$ and $\Pf(t)$ are
        in bijection. 
        \item The polynomials $\Pf(s)$ and $\Pf(t)$ have the same splitting
        behavior, i.e.\ the degrees of their irreducible factors, counted with multiplicities, are the same.
    \end{enumerate}
\end{coro}

\begin{proof}
Thanks to
    \cref{thm:Baer-correspondence}(2), there exist $\F_p$-linear isomorphisms
    $\alpha:V\rightarrow V'$ and $\beta:T\rightarrow T'$ such that for all
$u,v\in V$, 
    \begin{equation}\label{eqn:alpha-beta-iso}
        t(\alpha(u), \alpha(v)) = \beta(s(u,v)) . 
    \end{equation}
Then the  following map $\Cent(s) \to \Cent(t)$ is an isomorphism of $\F_p$-algebra:
    \[ 
        (X, Y, Z) \longmapsto (\kappa_{\alpha}(X), \kappa_{\alpha}(Y), \kappa_{\beta}(Z)) := (\alpha X\alpha^{-1}, \alpha Y \alpha^{-1}, \beta Z \beta^{-1}). 
    \] 
    Let $\varphi_s: \ff \to \Cent(s)$ and $\varphi_t: \ff \to \Cent(t)$ be
    field isomorphisms, so $s$ and $t$ are $\ff$-bilinear via $\varphi_s$ and
    $\varphi_t$, respectively. Since $s$ and $t$ are fully-nondegenerate, the
    maps $\pi_s : \Cent(s)\rightarrow\End_{\F_p}(T)$ and $\pi_t :
    \Cent(t)\rightarrow\End_{\F_p}(T')$ given by 
    \[
        (X,Y,Z)\longmapsto Z \ \textup{ and }\ (X',Y',Z')\longmapsto Z',
    \] 
    respectively, are injective, so $\im(\pi_s)\cong \im(\pi_t)\cong\ff$. Set
    $\psi_s = \pi_s\varphi_s$ and $\psi_t = \pi_t\varphi_t$, so $\sigma =
    \psi_t^{-1}\kappa_{\beta} \psi_s\in \Gal(\ff/\F_p)$. Thus, $\beta$ is
    $\ff$-semilinear: for all $u,v\in T$ and all $\lambda\in \ff$, 
    \[ 
        \beta(u+\lambda v) = \beta(u + \psi_s(\lambda) v) = \beta(u) + \kappa_{\beta}\psi_s(\lambda)\beta(v) = \beta(u) + \sigma(\lambda)\beta(v).
    \] 
    By a similar argument, $\alpha$ is also $\ff$-semilinear for some
    $\sigma'\in\Gal(\ff/F_p)$.
  
    Choose $\ff$-bases for $V$, $V'$, $T$, and $T'$. Since both $\alpha$ and
    $\beta$ are $\ff$-semilinear, there exist $\ff$-matrices $M$ and $N$ such
    that $\alpha$ is given by $(N,\sigma')$ and $\beta$ by $(M, \sigma)$
    relative to our choice of bases. If $u,v\in V$ are basis vectors, then
    \eqref{eqn:alpha-beta-iso} implies
    \begin{equation}\label{eqn:semi-linear}
        t(Nu, Nv) = t(\alpha(u), \alpha(v)) = \beta(s(u,v)) = M\sigma(s(u,v)).
    \end{equation}
    With $\lambda = \det(N)^2\in\ff^{\times}$, the equality in
    \eqref{eqn:semi-linear} implies \eqref{def:proj-semi-equiv} for the
    Pfaffians of $s$ and $t$ over $\ff$. Hence, (1) holds. The other statements (2) and (3) easily follow from (1).
\end{proof}

\section{Tensors and irreducible Pfaffians}\label{sec:tenspfaff}

In the study of elliptic $p$-groups $\GP_t(\ff)$, the number of the maximal totally isotropic
subspaces of $V$ is a fundamental invariant. We will show in this
section that there are four possible situations for the isotropic subspaces, distinguished by the adjoint algebra of $t$.
These four possibilities fit into two larger frameworks: $\tuB$ is either
isotropically-decomposable or isotropically-indecomposable. 
As a matter of fact, we prove this within the much larger framework of groups for which $\Pf(t)$ is an irreducible variety.
We 
eventually make an additional ``genericity'' assumption, which is satisfied by all
tensors coming from $E$-groups. 

\subsection{Pfaffians and isotropic subspaces}
In this section we present a number of results on isotropic subspaces with respect to an alternating tensor with irreducible Pfaffian.

\begin{lemma}\label{lem:max-dim}
 Let $t: V\times V \rightarrow T$ be an alternating $\alg$-tensor.  If $\Pf(t)\neq 0$ and
    $U\leq V$ is totally isotropic, then $2\dim_{\alg}U\leq \dim_{\alg}V$.
\end{lemma} 

\begin{proof} 
 From $\Pf(t)\neq 0$ it follows that $\dim_{\alg}V=2n$ for some $n\geq 1$.
    For a contradiction,  let $U$ be a totally isotropic subspace of $V$ with
    $\dim_{\alg}U\geq n+1$. Extend a basis of $U$ to a basis of $V$, and with
    this basis represent $t$ as the following  matrix of linear forms 
    \begin{align*} 
        \tuB = \begin{pmatrix} 
            0 & \tuM_1 \\ 
            -\tuM_1^{\tp} & \tuM_2
        \end{pmatrix} , \textup{ with } \tuM_1,\tuM_2\in\Mat_n(\alg[y_1,\ldots,y_d]_1), \ \tuM_2^{\tp}=-\tuM_2.
    \end{align*} 
Note that $\tuM_1$ contains a zero column since
    $\dim_{\alg}U\geq n+1$,  and so, for some
    $u\in\alg^\times$, the equality $\Pf(t)=u\det(\tuM_1)=0$ holds.  Contradiction.
\end{proof}

\begin{defn}
    Let $t: V\times V \rightarrow T$ be an alternating $\alg$-tensor.  The set
    of totally isotropic subspaces of $V$ of dimension $\frac{1}{2}\dim_{\alg}
    V$ is denoted $\ti{t}$.
\end{defn}

\begin{lemma}\label{lem:overlaps}
  Let $t:V \times V \rightarrow T$ be an alternating $\alg$-tensor such that $\Pf(t)$ is
    irreducible over $\alg$. If $U, W\in\ti{t}$, then either $U=W$ or
    $U\cap W=0$.
\end{lemma}

\begin{proof}
Fix $U, W\in\ti{t}$.
    Suppose $U\cap W\neq 0$ and $U\neq W$. Let $\mathcal{B}_{U\cap W}$ be a
    basis of $U\cap W$, and extend it to bases of $U$ and $W$, denoted by
    $\mathcal{B}_U$ and $\mathcal{B}_W$, respectively. Extend
    $\mathcal{B}_U\cup\mathcal{B}_W$ to a basis $\mathcal{B}$ of $V$. Let
    $\tuB$ be a matrix of linear forms associated to $t$ via this choice of basis for $V$. Since $\Pf(t)\neq 0$, we know that $\dim_{\alg}V=2n$ for some $n\geq 1$. It follows that there exist $n\times n$ matrices $\tuM_1$ and $\tuM_2$  of
    linear forms, with $\tuM_2$ skew-symmetric, such that
    \begin{align*} 
        \tuB = \begin{pmatrix} 
            0 & \tuM_1 \\ -\tuM_1^{\tp} & \tuM_2
        \end{pmatrix}.
    \end{align*}
By construction, there are two overlapping
    $n\times n$ blocks of zeros (corresponding to $U$ and $W$) along the diagonal of $\tuB$. In particular $\tuM_1$ is
    block-upper triangular, and therefore $\det(M_1)$ is not irreducible in $\alg[y_1,\ldots,y_d]$.
Since $\Pf(t)/\det(\tuM_1)$ is a non-zero  scalar,  we have reached a contradiction. 
\end{proof}

We remark that the proof of \Cref{lem:overlaps} may be further generalized to describe the
    intersection of $U,W\in\ti{t}$ for Pfaffians that split into a
    product of irreducible polynomials. 

\begin{defn}\label{def:equivalence}
Let $\mathcal{V}$ be a hypersurface in $\P^{d-1}_{\alg}$ of degree $n$ and let $I(\mathcal{V})\subseteq \alg[\bm{y}]$ be the ideal of $\mathcal{V}$. The following notions are defined as follows:
\begin{enumerate}[label=$(\arabic*)$]
\item if $\tuM, \tuM'\in\Mat_n(\alg[\bm{y}]_1)$ are such that $\det\tuM, \det\tuM'\in I(\mathcal{V})$, then 
\[
\tuM \sim \tuM' \ \Longleftrightarrow \textup{ there exist } X,Y\in\GL_n(\alg) \textup{ with } \tuM'=X\tuM Y.
\]
\item\label{it:det1} if $\tuM\in\Mat_n(\alg[\bm{y}]_1)$ is such that $\det\tuM\in I(\mathcal{V})$, then 
\[
[\tuM]=\{\tuM'\in\Mat_n(\alg[\bm{y}]_1) : \tuM\sim \tuM' \}.
\]
\item\label{it:det2} $\mathcal{L}_{\mathcal{V}}$ is the set of all $[\tuM]$ where $\tuM$ is as in \ref{it:det1}.
\item $\mathcal{L}_{\mathcal{V}}^{\mathrm{sym}}$ is the subset of $\mathcal{L}_{\mathcal{V}}$ consisting of all $[\tuM]$ with a symmetric representative.
\end{enumerate} 
\end{defn}

We remark that the relation $\sim$ defined above is the standard equivalence
relation considered in the study of determinantal varieties, cf.\
\cref{def:equivalence-tensor}. Moreover, it is clear that, for any two equivalent elements
$\tuM,\tuM'$, there exists $u\in\alg^\times$ such that $\det\tuM=u\det\tuM'$. When $n=d$, we refer to $\tuM$ as a \emph{cuboid}, which is \emph{smooth}, if $\det\tuM$ is a smooth polynomial in $\alg[\bm{y}]$. The name smooth cuboids is borrowed from \cite[Sec.\ 2]{Ng}; cf.\ \cref{rmk:Ng}.

\begin{prop}\label{prop:ti3d}
    Let $t:V\times V\rightarrow T$ be an alternating $\alg$-bilinear map with an
    irreducible Pfaffian and associated matrix $\tuB\in\Mat_{2n}(\alg[y_1,\dots,
    y_d]_1)$. If $|\ti{t}|\geq 2$, then there exists $\tuM\in
    \Mat_n(\alg[\bm{y}]_1)$ such that  $\tuB$ is pseudo-isometric to
    \[
    \begin{pmatrix}
        0 & \tuM \\ -\tuM^{\tp} & 0
    \end{pmatrix} \]
    and exactly one of the following holds:
    \begin{enumerate}[label=$(\arabic*)$]
        \item $|\ti{t}|>2$ and $[\tuM]\in\mathcal{L}_{\mathcal{V}}^{\mathrm{sym}}$. 
        \item $|\ti{t}|=2$ and $[\tuM]\in \mathcal{L}_{\mathcal{V}}\setminus \mathcal{L}_{\mathcal{V}}^{\mathrm{sym}}$.
    \end{enumerate}
\end{prop}
    
\begin{proof}
    Suppose $|\ti{t}|\geq 2$, and take $U,W\in\ti{t}$ to be   distinct. Then \Cref{lem:overlaps} yields $U\cap W=0$, so there exists $\tuM\in
    \Mat_n(\alg[\bm{y}]_1)$ and a choice of basis for $V$ such that 
    \[
    \tuB=\begin{pmatrix}
    0 & \tuM \\ -\tuM^{\tp} & 0
    \end{pmatrix}.
    \]
    Since $t$ has non-zero Pfaffian, $t$ is nondegenerate,
    so $\tuM$ defines a non-degenerate tensor $\tilde{t}:U\times W \rightarrow T$. Fix bases of
    $U$ and $W$, and let $W\rightarrow U$ be the linear map identifying the chosen 
    bases, written as $w\mapsto \overline{w}$.

    Let $X\in\ti{t}$, and assume without loss of generality that
    $X\cap U=0$.  Then both $W$ and $X$ are complements of $U$ in $V$, so there
    is $D\in\Mat_n(\alg)$ such that $X=\{w+D\overline{w} \mid w\in W\}$. Fix
    such a~$D$. Now $X$ is totally isotropic if and only if, for all $w,w'\in
    W$, the element $t(w+D\overline{w}, w'+D\overline{w'})$ is trivial. This
    happens if and only if, for all $w,w'\in W$, one has
    \begin{align}\label{eqn:ti-symmetric}
        0 &= 
        \begin{pmatrix} \overline{w}{}^{\tp} D^{\tp} & w^{\tp} \end{pmatrix}
        \begin{pmatrix} 0 & \tuM \\ -\tuM^{\tp} & 0 \end{pmatrix} 
        \begin{pmatrix} D\overline{w'} \\ w' \end{pmatrix}
        =-\overline{w}{}^{\tp}\tuM^{\tp}D w'+\overline{w}{}^{\tp}D^{\tp}\tuM w'. 
    \end{align}
    We conclude from \eqref{eqn:ti-symmetric}
    that $X$ is totally isotropic if and only if $D^{\tp}\tuM=\tuM^{\tp}
    D=(D^{\tp}\tuM)^{\tp}$, equivalently
    \begin{align} \label{eqn:symm-isom}
        \begin{pmatrix} 
            D^{\tp} & 0 \\ 0 & I 
        \end{pmatrix} \begin{pmatrix}
            0 & \tuM \\ -\tuM^{\tp} & 0
        \end{pmatrix}
        \begin{pmatrix} 
            D & 0 \\ 0 & I 
        \end{pmatrix} &= \begin{pmatrix}
            0 & D^{\tp}\tuM \\ -D^{\tp}\tuM & 0
        \end{pmatrix} .
    \end{align} 
        Call $\tuB'$ the matrix on the right hand side of \eqref{eqn:symm-isom}.
    If $D$ is invertible, then $X\neq W$ and $\tuB$ and $\tuB'$ are
    $\alg$-pseudo-isometric. This proves
    (1). If $D$ is not invertible, then let $w\in W\setminus\{0\}$ be such that
    $D\overline{w}=0$. it follows that 
    \[
        w=w+D\overline{w}\in X\cap W,
    \]
    and so \cref{lem:overlaps} yields $X=W$, which proves (2) and completes the
    proof.
\end{proof} 

\begin{coro}\label{cor:four-cases}
    Let $t:V\times V\rightarrow T$ be an
    alternating $\ff$-tensor such that $\Pf(t)$ describes an elliptic
    curve in $\P^2_{\ff}$. Then $\dim V=2\dim T=6$ and  
    $|\ti{t}|$ takes values in $\{0,1,2,q+1\}$.
\end{coro}

\begin{proof}
    An elliptic curve in $\P^2_{\ff}$ is defined by a homogeneous cubic polynomial in $3$ variables and therefore the dimensions of
    $V$ and $T$ are respectively $6$ and $3$. Moreover, if $|\ti{t}|>2$,
    then \cref{prop:ti3d} yields that $t$ can be represented by a symmetric
    matrix of linear forms. The claim follows from
    \cite[Prop.~4.10]{stanojkovski2019hessian}.
\end{proof}

\subsection{Adjoint algebras of half-generic tensors}

Before dealing directly with the specific situation where $n=3$ and
$\tuM=\J{E}{P}$, we turn to the more general case of half-generic tensors (see \cref{def:dim}), for which $|\ti{t}| \geq 1$. In this case, indeed, we have some more characteristic information given by the existence of maximal totally isotropic subspaces. As we will see, the adjoint algebra can tell exactly how many such subspaces are present and whether the equivalence class of $\tuB$ satisfies some additional symmetry constraints.
As supported by computational evidence, when $|\ti{t}|=0$, we expect to find that $\Adj(t) \cong \alg$.  
The adjoint algebras for half-generic tensors in the case when
$|\ti{t}|\geq 2$ are, thanks to \Cref{prop:ti3d}, already determined
without much additional work. We analyze the case where $|\ti{t}|=1$ in
Lemmas~\ref{lem:asymmetric-dim} and~\ref{lem:symmetric-dim}.

\begin{lem}\label{lem:adj-structure}
Let $\tuM\in\Mat_{n}(\alg[\bm{y}]_1)$ and write
    \begin{align*} 
        \tuB &= \begin{pmatrix} 
            0 & \tuM \\ 
            -\tuM^{\tp} & 0
        \end{pmatrix} .
    \end{align*} 
    Then the following holds:
    \begin{align*} 
        \Adj(\tuB) &= \left\{\left(\begin{pmatrix} 
            A_{11} & A_{12} \\ A_{21} & A_{22}
        \end{pmatrix}, \begin{pmatrix} 
            B_{11} & B_{12} \\ B_{21} & B_{22}
        \end{pmatrix} \right) ~:~ \begin{array}{c} 
            (A_{11}, B_{22}), (B_{11}, A_{22}) \in \Adj(\tuM), \\
            (A_{21}, -B_{21}) \in\Adj(\tuM, \tuM^{\tp}), \\ 
            (A_{12}, -B_{12}) \in\Adj(\tuM^{\tp}, \tuM)
        \end{array} \right\}.
    \end{align*} 
\end{lem} 

\begin{proof} 
    For $i,j\in [2]$, let $A_{ij},B_{ij}\in\Mat_n(\alg)$. Then 
    \begin{align*} 
        \begin{pmatrix} 
            A_{11}^{\tp} & A_{21}^{\tp} \\ A_{12}^{\tp} & A_{22}^{\tp}
        \end{pmatrix} \begin{pmatrix} 
            0 & \tuM \\ 
            -\tuM^{\tp} & 0
        \end{pmatrix} &= \begin{pmatrix} 
            0 & \tuM \\ 
            -\tuM^{\tp} & 0
        \end{pmatrix} \begin{pmatrix} 
            B_{11} & B_{12} \\ B_{21} & B_{22}
        \end{pmatrix}
    \end{align*} 
    implies that the following equations hold: 
    \[
        -A_{21}^{\tp}\tuM^{\tp}= \tuM B_{21}, \quad A_{11}^{\tp}\tuM = \tuM B_{22},\quad  -A_{22}^{\tp}\tuM^{\tp} = -\tuM^{\tp}B_{11}, \quad A_{12}^{\tp}\tuM = -\tuM^{\tp}B_{12} .    \qedhere 
    \]
\end{proof}

\begin{defn}\label{def:dim}
    A skew-symmetric matrix $\tuB\in \Mat_{2n}(\alg[\bm{y}]_1)$ is
    \emph{half-generic} if $\tuB$ is pseudo-isometric to
    $\left(\begin{smallmatrix} 0 & \tuM \\ -\tuM^{\tp} &
    \mathrm{S}\end{smallmatrix}\right)$ where $\tuM, \mathrm{S}\in
    \Mat_{n}(\alg[\bm{y}]_1)$ satisfy 
    \begin{align*} 
    \Adj(\tuM)\cong \alg \ \ \textup{ and } \ \ 
        \Adj(\tuM, \tuM^{\tp}) = \Adj(\tuM^{\tp}, \tuM) = \begin{cases} 
            \Adj(\tuM) & \text{ if } \tuM^{\tp} = \tuM, \\
            0 & \text{ otherwise}. 
        \end{cases} 
    \end{align*} 
    If $\tuM$ can be chosen to satisfy $\tuM^{\tp} = \tuM$, then $\tuB$ is \emph{symmetrically half-generic} and
    \emph{asymmetrically half-generic} if no such $\tuM$ exists. 
\end{defn}

\begin{remark} 
    In the situation described in \Cref{def:dim}, both $\Adj(\tuM)$ and
    $\Adj(\tuM^{\tp},\tuM)$ are as small as possible. This is what we expect in
    general for ``most'' linear determinantal representations, motivating the
    name ``half-generic''. Indeed, if $\tuM = \tuM_1y_1 + \ldots + \tuM_dy_d$
    with $\tuM_1,\ldots,\tuM_d\in\Mat_n(\alg)$, then the following hold:
    \begin{align*} 
        \Adj(\tuM) &= \bigcap_{i=1}^d \Adj(\tuM_i) & \Adj(\tuM, \tuM^{\tp}) 
        &= \bigcap_{i=1}^d \Adj(\tuM_i, \tuM_i^{\tp}) . 
    \end{align*} 
    Generically, both $\Adj(\tuM_i)$ and $\Adj(\tuM_i, \tuM_i^{\tp})$ define a
    codimension $n^2$ subspace of $\alg^{2n^2}$, so for sufficiently large $d$,
    generically $\Adj(\tuM)$ and $\Adj(\tuM, \tuM^{\tp})$ are given as in
    \Cref{def:dim}.
\end{remark}

Because half-genericity is defined for pseudo-isometry classes, it makes sense to extend \Cref{def:dim} to tensors, via the choice of a basis. Alternating tensors $t : V \times V \to T$
with $|\ti{t}| \geq 1$ and whose Pfaffian defines an elliptic curve are
half-generic, cf.\ \Cref{rem:elliptic-dimin}.

\begin{lem}\label{lem:asymmetric-dim} 
    Let $t:V\times V\rightarrow T$ be an asymmetrically half-generic
    $\alg$-tensor such that $|\ti{t}|=1$. Then $\Adj(t)$ is
    $*$-isomorphic to $\mathbf{O}_1(\alg)$.
\end{lem}

\begin{proof} 
    Let $\tuB\in \Mat_{2n}(\alg[\bm{y}])$ be associated with $t$ and in the form
    given in \cref{def:dim}. Since $t$ is asymmetrically half-generic,
    $\Adj(\tuM, \tuM^{\tp})=\Adj(\tuM^{\tp}, \tuM)=0$, and hence by a computation similar to that carried out in the proof of \cref{lem:adj-structure},  we have
    \begin{align} \label{eqn:asymm-adj}
        \Adj(\tuB) &= {\footnotesize \left\{\left(\begin{pmatrix} 
            aI_n & \Gamma \\ 0 & (a-k)I_n
        \end{pmatrix}, \begin{pmatrix} 
            (a-k)I_n & \Phi \\ 0 & aI_n 
        \end{pmatrix}\right) ~\middle|~ \begin{array}{c} 
            a, k \in \alg,\; \Gamma,\Phi\in\Mat_n(\alg), \\ 
            \Gamma^{\tp}\tuM + \tuM^{\tp} \Phi = k\mathrm{S} 
        \end{array}\right\} } .
    \end{align} 
    Now we consider the $\alg$-vector space 
    \begin{align*} 
        L = \left\{(\Gamma, \Phi, k)\in\Mat_n(\alg)^2\oplus \alg ~|~ \Gamma^{\tp}\tuM + \tuM^{\tp} \Phi = k\mathrm{S} \right\},
    \end{align*}
    which we will show to be trivial. First, we show that
    $\dim_{\alg}L\leq 1$. For this note that, if $(\Gamma, \Phi, 0)\in L$, then $(\Gamma, -\Phi)\in\Adj(\tuM^{\tp},
    \tuM)=0$.  Now, if $(\Gamma, \Phi, k), (\Gamma', \Phi', k')\in L$ with $k\neq 0$, then there    exists $\lambda\in \alg$ such that $\lambda k = k'$. It follows that $(\lambda \Gamma-\Gamma',\lambda \Phi-\Phi',0)\in L$ and therefore   
    $(\lambda \Gamma,
    \lambda \Phi, \lambda k) = (\Gamma', \Phi', k')$.  This proves the claim.
    
 Suppose now, for a contradiction, that $\dim_{\alg}L=1$,
    and let $(X, Y, 1)\in L$. Since $\mathrm{S}$ is skew-symmetric, also $(Y, X,
    -1)$ belongs to $L$. It follows that $(X+Y, -(X+Y))\in \Adj(\tuM^{\tp}, \tuM)$, equivalently $Y=-X$.
    Hence, $\dim_{\alg}\Adj(t)= 2$, and the following holds:
    \begin{align*} 
        \Adj(\tuB) &= {\footnotesize \left\{\left(\begin{pmatrix} 
            aI_n & kX \\ 0 & (a-k)I_n
        \end{pmatrix}, \begin{pmatrix} 
            (a-k)I_n & -kX \\ 0 & aI_n 
        \end{pmatrix}\right) ~\middle|~ \begin{array}{c} 
            a, k \in \alg,\; X\in\Mat_n(\alg),   \\ 
            X^{\tp}\tuM - \tuM^{\tp} X = \mathrm{S} 
        \end{array}\right\} } .
    \end{align*} 
    In particular, $\Adj(\tuB)$ contains 
    \begin{align*} 
        (\alpha, \alpha^*) &= \left(\begin{pmatrix} 0 & -X \\ 0 & I_n \end{pmatrix}, \begin{pmatrix} I_n & X \\ 0 & 0 \end{pmatrix}\right). 
    \end{align*} 
    It follows that $\alpha\alpha^* = \alpha^*\alpha = 0$. Hence, for all
    $u,v\in V$,  we compute
    \begin{align*} 
        t(\alpha(u), \alpha(v)) &= t(u, \alpha^*\alpha(v)) = 0 = t(\alpha\alpha^*(u), v) = t(\alpha^*(u), \alpha^*(v)). 
    \end{align*}
It follows that $\alpha(V)$ and $\alpha^*(V)$ are two distinct $n$-dimensional, totally isotropic
    subspaces, which contradicts
    $|\ti{t}|=1$.  We have proven that $\dim_{\alg}L=0$, from which we derive that $\Adj(\tuB) \cong \alg\cong\mathbf{O}_1(\alg)$. 
\end{proof}

Recall from \Cref{subsec:*-rings} that the notation $A\cong R\rtimes S$
indicates that $R$ is the Jacobson radical of the $*$-ring $A$ and that
$A/R\cong S$.

\begin{lem}\label{lem:symmetric-dim}
    Let $t:V\times V\rightarrow T$ be a symmetrically half-generic $\alg$-tensor
    such that $|\ti{t}|=1$. Assume that $\charac(\alg)\neq 2$.  Then
    $\Adj(t)$ is $*$-isomorphic to $ \alg\rtimes \mathbf{O}_1(\alg)$ and, if
    $(\alpha, \beta)\neq(0,0)$ is in the Jacobson radical of $\Adj(t)$, then $\ti{t} = \{\alpha(V)\}$.
\end{lem}

\begin{proof}
    Suppose, without loss of generality, that the matrix $\tuB$ associated with
    $t$ is given in the form from \cref{def:dim} with  $\tuM^{\tp} = \tuM$.  It
    follows from the assumptions that $\Adj(\tuM^{\tp},\tuM)=\Adj(\tuM)\cong
    \alg$ and in the same way as in \cref{lem:asymmetric-dim} we derive
    \begin{align*} 
        \Adj(\tuB) &= \left\{\left(\begin{pmatrix} 
            \Gamma_{11} & \Gamma_{12} \\ kI_n & \Gamma_{22} 
        \end{pmatrix}, \begin{pmatrix} 
            \Phi_{11} & \Phi_{12} \\ -kI_n & \Phi_{22} 
        \end{pmatrix}\right) ~\middle|~ \begin{array}{c} 
            k \in \alg,\; \Gamma_{ij},\Phi_{ij}\in\Mat_n(\alg), \\ 
            \Gamma_{11}^{\tp}\tuM - \tuM \Phi_{22} = -k\mathrm{S}, \\
            \Gamma_{22}^{\tp}\tuM - \tuM \Phi_{11} = k\mathrm{S}, \\
            \Gamma_{12}^{\tp} \tuM + \tuM \Phi_{12} = \mathrm{S}\Phi_{22} - \Gamma_{22}^{\tp}\mathrm{S}
        \end{array}\right\} .
    \end{align*} 
    Since $\Adj(\tuM)\cong \alg$,  the $\alg$-vector space 
    \begin{align*} 
        P &= \left\{(\Gamma, \Phi, k)\in\Mat_n(\alg)^2 \oplus \alg ~|~ \Gamma^{\tp}\tuM - \tuM \Phi = k\mathrm{S} \right\}
    \end{align*}
    has dimension $1$ or $2$. Indeed $(\Gamma,\Phi,0)\in P$ is equivalent to $(\Gamma, \Phi) \in
    \Adj(\tuM)$: the forward implication implies that $\dim_{\alg}P\leq 2$ while the reverse one ensures $\dim_{\alg}P\geq 1$. 

 We claim that  $\dim_{\alg}P=1$ and we prove so working by contradiction.   To this end, assume that $\dim_{\alg}P=2$ and let $\{(I_n, I_n, 0), (X, Y, 1)\}$ be a
    basis of $P$. Since $\mathrm{S}$ is skew-symmetric,  $(Y, X,
    1)$ also belongs to $P$, so $(X-Y, Y-X)\in \Adj(\tuM)$. Hence, there exists $\rho\in \alg$
    such that $X - Y = \rho I_n = Y - X$. This implies $2(X-Y) = 0$, and since
    $\charac(\alg)\neq 2$,  we derive $X=Y$. Define now $\mathrm{S}' =
    \mathrm{S}X + X^{\tp} \mathrm{S}$. With this, we have
    \begin{align*} 
        \Adj(\tuB) &= {\footnotesize\left\{\left(\begin{pmatrix} 
            aI_n - kX & C \\ kI_n & bI_n + kX
        \end{pmatrix}, \begin{pmatrix} 
            bI_n + kX & D \\ -kI_n & aI_n - kX
        \end{pmatrix}\right) ~\middle|~ \begin{array}{c} 
            a,b,k \in \alg,\; C,D\in\Mat_n(\alg), \\ 
            C^{\tp} \tuM + \tuM D = (a-b)\mathrm{S} - k\mathrm{S}' \\ 
            X^{\tp}\tuM - \tuM X = \mathrm{S}
        \end{array}\right\} }.
    \end{align*} 
    Now we consider the following vector space: 
    \begin{align*} 
        P' &= \left\{(C, D, c, d)\in\Mat_n(\alg)^2\oplus \alg^2 ~|~ C^{\tp}\tuM + \tuM D = c\mathrm{S} - d\mathrm{S}'\right\}. 
    \end{align*} 
Note that,  whenever $(C, D, c, 0) \in P'$, the element $(C, -D, c)$ belongs to $P$. It then follows from $\dim_{\alg} P = 2$ that $\dim_{\alg}P'\in \{2, 3\}$. If $\dim_{\alg}P'=2$, then a basis for $P'$ is $\{(I_{n}, -I_{n}, 0, 0), (X, -X, 1, 0)\}$ and 
    \begin{align*} 
        \Adj(\tuB) &= {\footnotesize\left\{\left(\begin{pmatrix} 
            aI_n & cI_n + (a-b)X \\ 0 & bI_n
        \end{pmatrix}, \begin{pmatrix} 
            bI_n & -cI_n + (b-a)X \\ 0 & aI_n
        \end{pmatrix}\right) ~\middle|~ \begin{array}{c} 
            a,b,c \in \alg, \\ 
            X^{\tp}\tuM - \tuM X = \mathrm{S}
        \end{array}\right\} }.
    \end{align*} 
   In particular, $\Adj(\tuB)$ contains
    the following element: 
    \begin{align} \label{eqn:problem-elts}
        (\alpha, \alpha^*) &= \left(\begin{pmatrix} 0 & -X \\ 0 & I_n \end{pmatrix}, \begin{pmatrix} I_n & X \\ 0 & 0 \end{pmatrix}\right). 
    \end{align} 
It follows that $\alpha(V)$ and $\alpha^*(V)$ are two distinct $n$-dimensional, totally isotropic
    subspaces, which contradicts
    $|\ti{t}|=1$.  This proves that 
    $\dim_{\alg}P' \neq 2$, so $\dim_{\alg}P' = 3$.

    Let now $\{(I_n, -I_n, 0, 0), (X, -X, 1, 0), (Y, Z, \ell, 1)\}$ be a basis of
    $P'$. Since both $\mathrm{S}$ and $\mathrm{S}'$ are skew-symmetric, we have $(Z, Y,
    -\ell, -1) \in P'$, so $(Y+Z, Y+Z, 0, 0)\in P'$.  The characteristic of $\alg$ being different from $
    2$, this implies that $Z=-Y$. Thus, if $C, D\in \Mat_n(\alg)$ and $a,b,k\in \alg$ are such that 
    \[
        C^{\tp} \tuM + \tuM D = (a-b)\mathrm{S} - k\mathrm{S}',
    \]
   then the following holds: for some $c\in \alg$,
    \begin{align*} 
        C &= cI_n + (a-b-k\ell)X + kY = -D.
    \end{align*}
    Thus,  taking $b=1$ and $a=c=k=0$, we see that the element in~\eqref{eqn:problem-elts} is
    also contained in $\Adj(\tuB)$; contradiction.
 In particular $\dim_{\alg}P' \neq 3$, which yields that $\dim_{\alg}P \neq 2$ and, consequently, that
    $\dim_{\alg}P = 1$.  
    
We conclude by observing that  $\Adj(t) \cong \alg\rtimes \mathbf{O}_1(\alg)$
    because  
    \begin{align} \label{eqn:adj-sym}
        \Adj(\tuB) &= \left\{\left(\begin{pmatrix} 
            aI_n & bI_n \\ 0 & aI_n 
        \end{pmatrix}, \begin{pmatrix} 
            aI_n & -bI_n \\ 0 & aI_n 
        \end{pmatrix}\right) ~\middle|~ a,b\in \alg \right\} .
    \end{align} 
    We also see from \eqref{eqn:adj-sym} that if $0\neq(\alpha,-\alpha)$ is in the radical of $\Adj(\tuB)$, then $\alpha(V)\in\ti{t}$. 
\end{proof}

\begin{prop}\label{prop:adj-X-S}
    Let $t : V \times V \rightarrow T$ be a $\alg$-tensor with irreducible
    Pfaffian. If $\Adj(t)$ is $*$-isomorphic to either $\mathbf{X}_1(\alg)$ or
    $\mathbf{S}_2(\alg)$, then $|\ti{t}|= 2$ or $|\ti{t}|> 2$, respectively.
\end{prop}

\begin{proof}
    Let $\phi : \Adj(t)\to A$ be a $*$-isomorphism, where $A$
    is either $\mathbf{X}_1(\alg)$ or $\mathbf{S}_2(\alg)$. 
    
If $A=\mathbf{X}_1(\alg)$ and $(X, Y) =
    \phi^{-1}((1,0))$, then \cref{thm:simple-star-algebras} yields that $XY=0=YX$ and  $X+Y=I_{2n}$.  Then $U=XV$ and $W=YV$ are totally isotropic subspaces and satisfy $U+W=V$.  In particular,  $|\ti{t}|\geq 2$.  To show that equality holds, we assume for a contradiction that $U'\in\ti{t}\setminus\{U,W\}$. Then thanks to \cref{lem:overlaps}, we write $V=U\oplus U'$. Taking $\alpha:V\rightarrow U$ and $\beta:V\rightarrow U'$ to be projections such that $\alpha_{|U}=\id_U$ and $\beta_{|U'}=\id_{U'}$, it follows that $(\alpha,\beta)\in\Adj(t)$ and so the dimensions of $A$ and $\Adj(t)$ do not coincide.  
    
   Assume now that $A=\mathbf{S}_2(\alg)$ and define 
   \[
(X_1,Y_1)=\varphi^{-1}\left(
\begin{pmatrix}
1 & 0 \\ 0 & 0 
\end{pmatrix} 
\right)  
\ \ \textup{ and } \ \ 
(X_2,Y_2)=\varphi^{-1}\left(
\begin{pmatrix}
1 & 1 \\ 1 & 1 
\end{pmatrix} 
\right). 
   \]
\cref{thm:simple-star-algebras} yields that  $X_1Y_1=Y_1X_1=0=X_2Y_2=Y_2X_2$. Therefore 
 $X_1V$, $X_2V$,
    and $Y_1V$ are totally isotropic subspaces. Moreover,  as $X_1+Y_1$, $X_1+X_2$, and $X_2+Y_1$ are invertible, \cref{lem:overlaps} ensures that they are distinct. In particular $|\ti{t}| > 2$.
\end{proof}

In the following, all isomorphisms concerning adjoint algebras are taken as
$*$-algebras.

\begin{thm}\label{thm:reducible-adjoints}
    Let $t : V \times V \rightarrow T$ be a full, half-generic $\alg$-tensor
    with irreducible Pfaffian. Then $\Cent(t)\cong \alg$, and the following
    hold.
    \begin{enumerate}[label=$(\arabic*)$] 
        \item\label{thm-part:asym} $|\ti{t}|=1$ and $t$ asymmetrically
        half-generic imply $\Adj(t)\cong \mathbf{O}_1(\alg)$.
        \item\label{thm-part:sym} $|\ti{t}|=1$, $t$ symmetrically
        half-generic, and $2\alg=\alg$ imply $\Adj(t)\cong \alg \rtimes
        \mathbf{O}_1(\alg)$.
        \item\label{thm-part:2} $|\ti{t}|=2$ if and only if $\Adj(t)
        \cong \mathbf{X}_1(\alg)$.
        \item\label{thm-part:q+1} $|\ti{t}|>2$ if and only if $\Adj(t)
        \cong \mathbf{S}_2(\alg)$.
    \end{enumerate}
\end{thm}

\begin{proof}
    Let $\tuB\in\Mat_{2n}(\alg[\bm{y}]_1)$ be the matrix associated with $t$.
    Since the Pfaffian is non-zero,  we write $\dim_{\alg}V=2n$. If
    $|\ti{t}|=1$, then  Lemmas~\ref{lem:asymmetric-dim}
    and~\ref{lem:symmetric-dim} settle  \ref{thm-part:asym}--\ref{thm-part:sym}.  If $|\ti{t}|=2$, we know from \cref{prop:ti3d},
    that $\tuB$ is asymmetrically half-generic.  In this case
    \cref{lem:adj-structure} yields
    \begin{align} \label{eqn:X-adj}
        \Adj(\tuB) &= \left\{\left(\begin{pmatrix} aI_n & 0 \\ 0 & bI_n 
        \end{pmatrix}, \begin{pmatrix} bI_n & 0 \\ 0 & aI_n \end{pmatrix} 
        \right) ~\middle|~ a,b\in\alg \right\} \cong \mathbf{X}_1(\alg) .
    \end{align} 
    This proves the forward direction of \ref{thm-part:2}. If, on the other
    hand, $|\ti{t}|>2$, then $\tuB$ is symmetrically half-generic by
    \cref{prop:ti3d}. By \cref{lem:adj-structure}, we have
    \begin{align}\label{eqn:S-adj} 
        \Adj(\tuB) &= \left\{\left(\begin{pmatrix} aI_n & bI_n \\ cI_n & dI_n 
        \end{pmatrix}, \begin{pmatrix} dI_n & -bI_n \\ -cI_n & aI_n 
        \end{pmatrix}\right) ~\middle|~ a,b,c,d\in\alg\right\}\cong\mathbf{S}_2(\alg) . 
    \end{align} 
    Thus, the forward direction of \ref{thm-part:q+1} follows
    from~\eqref{eqn:S-adj}. The reverse directions for \ref{thm-part:2} and
    \ref{thm-part:q+1} follow from \cref{prop:adj-X-S}.

We conclude by showing that   $\Cent(t)\cong \alg$.   Because $\Pf(t)$ is irreducible and $t$ is full, $t$ is fully-nondegenerate. Thus, \Cref{lem:comm-cent}
    ensures that $\Cent(t)$ is commutative. Moreover, from~\cite[Th.~A]{BMW/20},  we know
    that $\Cent(t)$ embeds into the center of $\Adj(t)$ which is,  in the four
    cases of this theorem, always isomorphic to $\alg$.  This ends the proof.
\end{proof}

\section{Determinantal representations of cubics}\label{sec:determinantal}

Let $E$ be an elliptic curve in $\P^2_{\alg}$ with identity element $\cor{O}$.
Let $E(\alg)$ denote the $\alg$-rational points of $E$ and $\oplus$ the addition
on $E$. If $P\in E(\alg)$ and $m$ is a positive integer, we write $\ominus P$
for the opposite of $P$ and $[m]P\in E(\alg)$ and $E[m]$, respectively, for the
sum of $m$ copies of $P$ and the $m$-torsion subgroup of $E$. We write $\Aut(E)$
for the group of automorphisms of $E$ as a projective curve, and
$\Aut_{\cor{O}}(E)$ for the automorphism group of the elliptic curve $E$. In
particular, $\Aut_{\cor{O}}(E)$ comprises the elements of $\Aut(E)$ that fix
$\cor{O}$. We let $\Div(E)$ be the Weil divisor group of $E$ and $\Div^0(E)$ the
degree $0$ part of the divisor group of $E$. We denote the class of $D$ in the
Picard group $\Pic(E)$ by $[D]$ and write $\Pic^0(E)$ for the degree $0$ part of
$\Pic(E)$. If $D\in\Div(E)$, denote by $\mathcal{L}(D)$ the sheaf of $D$, which
is also commonly denoted by $\cor{O}_E(D)$. Our notation is mostly adherent to
the one from \cite[Ch.~II]{silverman} and we refer the reader to
\cite{Hartshorne,silverman} for the geometric background of
\cref{sec:determinantal,sec:ABC}.

\subsection{Divisors on elliptic curves}\label{sec:divisors}

 In this section, we recall some basic facts on divisors.

\begin{lem}[{\cite[Prop.~II.6.13]{Hartshorne}}]\label{lem:divlinbdl}
Let $D, D'\in\Div(E)$. Then the following are equivalent: 
\begin{enumerate}[label=$(\arabic*)$]
\item $[D]=[D']$,
\item $\cor{L}(D)$ and $\cor{L}(D')$ are isomorphic.
\end{enumerate} 
\end{lem}

Let $\cor{L}_0$ denote the collection of isomorphism classes of non-effective degree $0$ line bundles on $E$. 
As a consequence of \cref{lem:divlinbdl} one can easily derive that the map
\begin{equation}\label{eq:D}
\Pic^0(E)\setminus\{0\} \longrightarrow \cor{L}_0,\quad [D]\longmapsto \cor{L}(D)
\end{equation}
is a well-defined bijection.
Moreover, \cite[Prop.~III.3.4]{silverman} ensures that the following map is a well-defined bijection: 
\begin{equation}\label{eq:pic0}
E\setminus \{\cor{O}\}\longrightarrow\Pic^0(E)\setminus\{0\},  \quad P \longmapsto [P- \cor{O}].
\end{equation}

\begin{lem}[Abel's Theorem, cf.\ {\cite[Cor.~III.3.5]{silverman}}]\label{lem:abel}
Let $D=\sum_{P\in E(\alg)}n_PP\in \Div(E)$. Then $[D]=0$ if and only if $\sum_{P\in E(\alg)}n_P=0$ and $\bigoplus_{P\in E(\alg)}[n_P]P=0$.
\end{lem}

\subsection{Linear representations and non-effective line bundles}\label{sec:beauville}

This short section collects a few results on equivalence classes of determinantal representations of cubics from \cite{Ishitsuka/17}, which build upon the seminal paper \cite{Beauville00} of Beauville. 

\begin{lem}[{\cite[Th.~5.2]{Ishitsuka/17}}]\label{prop:bundles}
Let $C$ be a smooth genus $1$ curve in $\P^2_{\alg}$. Then there is a natural bijection between the following sets:
\begin{enumerate}[label=$(\arabic*)$]
\item the set $\mathcal{L}_C$ as given in \cref{def:equivalence}\ref{it:det2},
\item the set of isomorphism classes of non-effective line bundles of degree $0$ on $C$.
\end{enumerate}
\end{lem}

\begin{defn}\label{def:ps}
Let $\cor{C}$ denote the collection of smooth genus $1$ curves in $\P^2_{\alg}$ and write 
\[
\cor{L}_{\cor{C}}=\bigcup_{C\in \cor{C}}\cor{L}_{C}. 
\]
For each choice of $\tuM \in\cor{L}_{\cor{C}}$, the notation $\ps{\tuM}$ is used for the orbit of $\tuM$ with respect to the action of $\Gamma=\GL_3(\alg)\times \GL_3(\alg)\times\GL_3(\alg)$ on $\cor{L}_{\cor{C}}$ that is defined by
\begin{align*}
\Gamma\times 
\cor{L}_{\cor{C}} & \longrightarrow \cor{L}_{\cor{C}}, \\
((X,Y,Z), \tuM(\bm{y})) & \longmapsto X^{\tp}\tuM(Z\bm{y})Y.
\end{align*}
Moreover,  $\mathcal{M}_{\mathcal{C}}$ denotes the collection of all $\ps{\tuM}$
where $\tuM\in\cor{L}_{\cor{C}}$.
\end{defn}

\begin{prop}\label{prop:M-bundles}
Let $\cor{C}$ denote the collection of smooth genus $1$ curves in $\P^2_{\alg}$. 
Let $[\tuM],[\tuM']\in\cor{L}_\cor{C}$ and write $C$ and $C'$ for the curves defined by $\det\tuM=0$ and $\det\tuM'=0$, respectively. Let, moreover, $\cor{L}$ and $\cor{L}'$ be degree $0$ non-effective line bundles on $C$ resp.\ $C'$ associated to $\tuM$ and $\tuM'$ as in \cref{prop:bundles}. Then the following are equivalent: 
\begin{enumerate}[label=$(\arabic*)$]
\item $\ps{\tuM}=\ps{\tuM'}$,
\item there exists a linear isomorphism $\gamma : C\rightarrow C'$ such that $\gamma^*\cor{L}'=\cor{L}$.
\end{enumerate}
\end{prop}

\begin{proof}
Rephrasing \cref{def:ps}, one has that  $\ps{\tuM}=\ps{\tuM'}$ if and only if there exists $Z\in\GL_3(\alg)$ such that
$[\tuM'({\bm y})]=[\tuM(Z\bm{y})]$.  By calling $\gamma$ the map $C\rightarrow C'$ that is induced by $Z$, it follows from \cref{prop:bundles} that $\ps{\tuM}=\ps{\tuM'}$ if and only if $\gamma^*\cor{L}'=\cor{L}$.
\end{proof}

\subsection{Explicit Weierstrass representations}\label{sec:ishitsuka}

\noindent
Until the end of the present section, let $a,b\in \alg$ with $4a^3+27b^2\neq 0$ and let $E$ denote the elliptic curve 
\begin{equation}\label{eq:E-sWe}
E\ :\ y^2z=x^3+axz^2+bz^3.
\end{equation}
In this paper, when talking of a \emph{Weierstrass equation} of  a curve $E$ we will mean an equation of the form \eqref{eq:E-sWe}, commonly referred to as a short Weierstrass equation of $E$.  
In the following definition, the matrix $\J{E}{P}$ is equivalent, in the sense of \cref{def:equivalence}, to the matrix $M_P$ from \cite[Ex.~7.6]{Ishitsuka/17}.

\begin{defn}\label{def:JP} 
 Let $P=(\lambda,\mu,1)\in E(\alg)$. Then $\J{E}{P}\in\Mat_3(\alg[x,y,z])$ is defined by
\[
\J{E}{P}=\J{E}{P}(x,y,z)=\begin{pmatrix}
x-\lambda z & y-\mu z & 0 \\
y+\mu z & \lambda x +(a+\lambda^2)z & x \\
0 & x & -z 
\end{pmatrix}.
\]
Moreover, the matrix $\tuB_{E,P}=\tuB_{E,P}(x,y,z)\in \Mat_6(\alg[x,y,z])$ is defined as 
\[
\tuB_{E,P}=\begin{pmatrix}
0 & \J{E}{P} \\
-\J{E}{P}^{\tp} & 0
\end{pmatrix}
\]
and the group $\GP_{\tuB_{E,P}}(\alg)$ is denoted $\GP_{E,P}(\alg)$.
\end{defn}

Note that all (projective) points in $E(\alg)\setminus\{\cor{O}\}$ are of the form $(\lambda,\mu,1)$ as in \cref{def:JP}.
The following result is the combination of Theorem 5.2, i.e.\ this paper's \cref{prop:bundles}, and Proposition 7.1 from \cite{Ishitsuka/17}; cf.\ also \cite[Ex.~7.6]{Ishitsuka/17}. For an alternative reference, see for instance \cite[Th.~1]{RaviTri14}.

\begin{prop}\label{thm:ishitsuka}
Assume $6\alg=\alg$. Then the following hold:
\begin{enumerate}[label=$(\arabic*)$]
\item the following map is a well-defined bijection
\[
E(\alg)\setminus\{\cor{O}\}\longrightarrow \cor{L}_E, \quad P\longmapsto [\J{E}{P}].
\]
\item for $P\in E(\alg)\setminus\{\cor{O}\}$, one has 
\[
[\J{E}{P}]\in \cor{L}_E^{\mathrm{sym}}\ \Longleftrightarrow \ P\in  E[2](\alg).
\]
\end{enumerate}
\end{prop}

\begin{remark}\label{rmk:JP}
Let $P=(\lambda, \mu,1)\in E(\alg)$. Then the following hold:
\begin{enumerate}[label=$(\arabic*)$]
\item $\J{E}{P}^{\tp}=\J{E}{\ominus P}$.
\item\label{it:JP} $\J{E}{P}$ is symmetric if and only if $\mu=0$, equivalently $P$ has order $2$ in $E(\alg)$.
\item if $\mu=0$, then $\J{E}{P}$ is equivalent to the following ``Hessian matrix'': 
\[
\mathrm{H}_{E,P}=\begin{pmatrix}
3\lambda x+az & y & a\lambda z+ax+3bz \\
y & x-\lambda z & -\lambda y\\
a\lambda z +ax+3bz & -\lambda y & a\lambda x+3\lambda bz+3bx-a^2z
\end{pmatrix}.
\]
This matrix corresponds to one of the three solutions to Hesse's system for $E$; cf.\ \cite{Hesse/44} and, with direct connection to this paper's work, \cite[Eq.~(1.6)]{stanojkovski2019hessian}.
\end{enumerate}
\end{remark}

\begin{remark}[Flex points for linearity]\label{rmk:flex}
An elliptic curve $\tilde{E}$ in $\P^2_{\alg}$ is defined by a smooth cubic in $\alg[y_1,y_2,y_3]$ which, however, need not be in short Weierstrass form.  For a $\alg$-linear change of coordinates to exist in order to express $\tilde{E}$ by a short Weierstrass equation, a sufficient condition is for $\tilde{E}$ to have a flex point over $\alg$. This is explicitly explained in \cite[Sec.~4.4]{Cremona/03} and accounts of this transformation can also be found in \cite[Prop.~III.3.1]{silverman} and \cite[Sec.~1.3]{SilvermanTate}. Via this change of coordinates, the flex point is mapped to the point at infinity $\mathcal{O}=(0:1:0)$, given in projective coordinates, which is also taken to be the identity for the group law on $\tilde{E}$.  Once this identification is made, it is a classical result that the collection of flex points in $\tilde{E}(\alg)$ coincides with $\tilde{E}[3](\alg)$ where the unique element of order $1$ is precisely $\cor{O}$; cf.\ \cite[Ex.~5.37]{Fulton69}. 
\end{remark}

\section{Proofs of Theorems \ref{mainthm:sat},  \ref{mainthm:aut}, and \ref{mainthm:adj}}\label{sec:ABC}

Relying on a number of techniques including the employment of Lie algebras (via the Baer correspondence) and results on the realization of elliptic curves as zero sets of determinants of $3\times 3$ matrices of linear forms, in Sections \ref{sec:AdjTh}, \ref{sec:proof-sat}, and \ref{sec:proof-aut} we prove Theorems \ref{mainthm:adj}, \ref{mainthm:sat}, and \ref{mainthm:aut} in the forms of \cref{cor:adj-cent-elliptic}, \cref{th:awesome}, and \cref{cor:pseudo}, respectively. To this end, we let $E$ denote the elliptic curve 
\[
E\ :\ y^2z=x^3+axz^2+bz^3, \quad a,b\in \alg \textup{ with } 4a^3+27b^2\neq 0.
\]

\subsection{Adjoint algebras for $E$-groups and the proof of \cref{mainthm:adj}}\label{sec:AdjTh}

We are now ready to describe the structure of the adjoint algebras of
$E$-groups coming from an isotropically decomposable $\tuB$ with upper-right corner $\tuM=\J{E}{P}$.

\begin{prop}\label{prop:JP-adj}
  Assume $\alg=2\alg$, and let $P=(\lambda,\mu,1)\in E(\alg)$. Then one has
    \begin{align*} 
        \Adj(\J{E}{P})&=\{(kI_3,kI_3) ~|~ k\in\alg\}, \ \ \textup{ and } \ \ \Adj(\J{E}{P},\J{E}{P}^{\tp}) = \begin{cases} 
            \Adj(\J{E}{P}) & \text{if } \mu=0, \\
            0 & \text{otherwise}. 
        \end{cases} 
    \end{align*} 
\end{prop}

\begin{proof} 
    Let $\bm{z} = (z_1,z_2,z_3)$ be variables. The adjoint algebra $\Adj(\J{E}{P})$
    is determined by a linear system of $27$ equations in $18$ variables. We
    describe an $M(\bm{z})\in\Mat_{18\times 27}(\alg[\bm{z}])$ such that the
    left kernel of $M(\lambda, \mu, a+\lambda^2)$ is in bijection with
    $\Adj(\J{E}{P})$. Write $\J{E}{P} = Ax + By + Cz$ for $A,B,C\in\Mat_3(\alg)$ computable from \cref{def:JP}, and
    define 
    \begin{align} \label{eqn:original-system}
        M(\bm{z}) &= {\footnotesize \begin{bmatrix} 
            -A & 0 & 0 & -B & 0 & 0 & -C & 0 & 0 \\
            0 & -A & 0 & 0 & -B & 0 & 0 & -C & 0 \\
            0 & 0 & -A & 0 & 0 & -B & 0 & 0 & -C \\
            E_{11} & z_1 E_{21} + E_{31} & E_{21} & E_{21} & E_{11} & 0 & -z_1 E_{11} -z_2 E_{21} & z_2 E_{11} + z_3 E_{21} & -E_{31} \\
            E_{12} & z_1 E_{22} + E_{32} & E_{22} & E_{22} & E_{12} & 0 & -z_1 E_{12} -z_2 E_{22} & z_2 E_{12} + z_3 E_{22} & -E_{32} \\
            E_{13} & z_1 E_{23} + E_{33} & E_{23} & E_{23} & E_{13} & 0 & -z_1 E_{13} -z_2 E_{23} & z_2 E_{13} + z_3 E_{23} & -E_{33} 
        \end{bmatrix}},
    \end{align}
 where $E_{ij}$ is the $3\times 3$ matrix with $1$ in the $(i,j)$-entry and $0$
    elsewhere. 
    
    By performing Gaussian elimination over $\alg[{\bm z}]$ on the columns of $M({\bm z})$, one can conclude that the left kernel of $M({\bm z})$ is
    $1$-dimensional, regardless of the values of $\lambda$, $\mu$, or $a$. The
    computations have been carried out in \textsf{SageMath}~\cite{SageMath} and
    \textsf{Magma}~\cite{magma/97}.  Now, since the adjoint algebra is unital,
    the left kernel of $M(\bm{z})$ being $1$-dimensional implies that $\Adj(\J{E}{P})\cong \alg$. 
    
    For $\Adj(\J{E}{P}, \J{E}{P}^{\tp})$,  \cref{rmk:JP}\ref{it:JP} ensures
    that, if $\mu=0$, then $\Adj(\J{E}{P}, \J{E}{P}^{\tp})=\Adj(\J{E}{P})$. Now
    assume that $\mu\neq 0$. A matrix $M'(\bm{z})$ whose left kernel defines
    a basis for $\Adj(\J{E}{P}, \J{E}{P}^{\tp})$, at
    $\bm{z}=(\lambda,\mu,a+\lambda^2)$, is obtained from $M(\bm{z})$ by
    replacing the three blocks $-C$ with $-C^{\tp}$. 
By performing Gaussian elimination over $K[{\bm z}]$ one can show that  $M'(\bm{z})$ has full
    rank, implying 
    $\Adj(\J{E}{P}, \J{E}{P}^{\tp})=0$. These computations have also been
    carried out in \textsf{SageMath} and \textsf{Magma}.
\end{proof}

\begin{remark}\label{rem:elliptic-dimin}
    The consequence of \Cref{prop:JP-adj}, combined with \cref{rmk:JP}, is that for all skew-symmetric
    matrices $\mathrm{S}\in \Mat_3(\alg)$, the matrix $\tuB =
    \left(\begin{smallmatrix} 0 & \J{E}{P} \\ -\J{E}{P}^{\tp} & \mathrm{S}
    \end{smallmatrix}\right)$ is half-generic. 
\end{remark}

\begin{coro}\label{cor:adj-cent-elliptic}
    Let $t : V \times V \rightarrow T$ be a fully-nondegenerate alternating
    $\alg$-tensor whose Pfaffian defines a smooth cubic $E$ in $\P^2_{\alg}$ with a flex point $\cor{O}\in
    E(\alg)$.  Assume $6\alg =
    \alg$. Then the following hold.
    \begin{enumerate}[label=$(\arabic*)$] 
        \item\label{cor-part:2} $|\ti{t}| = 2$ if and only if 
        $\Adj(t)\cong \mathbf{X}_1(\alg)$.
        \item\label{cor-part:q+1} $|\ti{t}| > 2$ if and only if 
        $\Adj(t)\cong \mathbf{S}_2(\alg)$.
    \end{enumerate}
\end{coro} 

\begin{proof} 
    Since $t$ is fully-nondegenerate with a Pfaffian defining a smooth cubic
    in $\P^2_{\alg}$, one has $\dim_{\alg}V = 6$ and $\dim_{\alg}T=3$. Let
    $\tuB\in \Mat_{6}(\alg[y_1,y_2,y_3]_1)$ be associated with $t$. 
    
    First we assume $|\ti{t}|\geq 2$. Without loss of generality, assume $\tuB$
    is isotropically decomposed with top right $3\times 3$ block equal to
    $\tuM$. Since $\charac(\alg)\notin \{2,3\}$ and $\Pf(\tuB)$ has a
    $\alg$-rational flex, \cref{rmk:flex} together with \cref{prop:M-bundles} ensure the existence of a pair $(E,P)$ such that $\ps{\tuM} = \ps{\J{E}{P}}$. Therefore, $\Adj(\tuM) \cong
    \Adj(\J{E}{P})$, and by \cref{prop:JP-adj}, the tensor $t$ is half-generic. The forward directions for both   \ref{cor-part:2} and \ref{cor-part:q+1} follow from \Cref{thm:reducible-adjoints}. The reverse directions  for both follow from \cref{prop:adj-X-S}. 
\end{proof} 

\begin{proof}[Proof of \cref{mainthm:adj}]
If $t$ is the tensor associated to the matrix $\tuB$ from \cref{mainthm:adj}, then by \cref{cor:adj-cent-elliptic} the following equivalences hold
\[
|\ti{t}| = 2\ \Longleftrightarrow\ \Adj(t)\cong \mathbf{X}_1(\alg), \quad |\ti{t}|> 2\ \Longleftrightarrow\ \Adj(t)\cong \mathbf{S}_2(\alg).
\]
Relying on $\tuB$ being (a)symmetrically half-generic, as is done in the proof
\cref{thm:reducible-adjoints}, combining \cref{prop:JP-adj} with \cref{rmk:JP}
we deduce that
\[
|\ti{t}| = 2\ \Longleftrightarrow\ P\in E(\alg)\setminus E[2](\alg), \quad |\ti{t}|> 2\ \Longleftrightarrow\ P\in E[2](\alg)\setminus\{\cor{O}\}. \qedhere
\] 
\end{proof}

\subsection{Isomorphism testing via isogenies and the proof of
\Cref{mainthm:sat}}\label{sec:proof-sat}

In this section, we give necessary and sufficient conditions for two groups of
the form $\GP_{E,P}(\ff)$ and $\GP_{E',P'}(\ff)$ to be isomorphic. 

\begin{lemma}\label{lem:swap}
Define the following matrices in $\GL_3(\alg)$:
\[
X_0=\begin{pmatrix}
-1 & 0 & 0 \\ 0 & 1 & 0 \\ 0 & 0 & 1
\end{pmatrix}\ \ \textup{ and } \ \ 
Z_0=\begin{pmatrix}
-1 & 0 & 0 \\ 0 & 1 & 0 \\ 0 & 0 & -1
\end{pmatrix}.
\]
For every elliptic curve $E$ in short Weierstrass form over $\alg$ and every
$P\in E(\alg)\setminus\{\cor{O}\}$, the following equality holds:
\[
\begin{pmatrix}
0 & X_0 \\ X_0 & 0
\end{pmatrix}
\tuB_{E,P}(Z_0{\bm y})
\begin{pmatrix}
0 & X_0 \\ X_0 & 0
\end{pmatrix} = \tuB_{E,P}({\bm y}).
\]
\end{lemma}

For the next result, recall that a group
$G$ acts \emph{$k$-transitively} on a set $X$ if its induced action $g\cdot
(x_1,\dots, x_k) = (g(x_1), \dots, g(x_k))$ on the subset of $X^k$ of all
elements with pairwise distinct entries is transitive. The next result is an easy consequence of \cite[Prop.~4.10]{stanojkovski2019hessian}.

\begin{lemma}\label{lem:psi-M}
Let $E$ be an elliptic curve in short Weierstrass form over the field $\alg$ and let  $P\in E[2](\alg)\setminus\{\cor{O}\}$.  Define, moreover, $\psi:\GL_2(\alg)\rightarrow\GL_6(\alg)$ by 
\[
\begin{pmatrix}
a & b \\ c & d
\end{pmatrix} \longmapsto 
\begin{pmatrix}
aI_3 & bI_3 \\ cI_3 & dI_3
\end{pmatrix}.
\]
Then $\psi(\GL_2(\alg))$ acts $2$-transitively on $\ti{\tuB_{E, P}}$.
\end{lemma}

\begin{thm}
\label{th:awesome}
    Let $E$ and $E'$ be elliptic curves in $\P^2_{\ff}$ given by Weierstrass
    equations, and let $P\in E(\ff)\setminus\{\cor{O}\}$ and $P'\in
    E'(\ff)\setminus\{\cor{O}'\}$. Assume $\charac(\ff)=p\geq 5$.  Then the
    following are equivalent.
    \begin{enumerate}[label=$(\arabic*)$]
        \item The $\ff$-Lie algebras $\LA_{E,P}(\ff)$ and $\LA_{E',P'}(\ff)$ are
        isomorphic.
        \item The set $\pseudo_F(\tuB_{E,P},\tuB_{E',P'})$ is nonempty.
        \item There exists an isomorphism $\varphi:E\rightarrow E'$ of elliptic
        curves such that $\varphi(P)=P'$.
    \end{enumerate}
\end{thm}

\begin{proof}
Let $t:U\times W\rightarrow T$ and $t':U'\times W'\rightarrow T'$ be the
$\ff$-tensors defined by $\J{E}{P}$
and $\J{E'}{P'}$, respectively. Up to postcomposing with the automorphisms from
\cref{lem:swap} or \cref{lem:psi-M}, the Lie algebras $\LA_{E,P}(\ff)$ and $\LA_{E',P'}(\ff)$ are
isomorphic if and only if there exists an isomorphism $\LA_{E,P}(\ff)\rightarrow\LA_{E',P'}(\ff)$ mapping $U$ to $U'$ and $W$ to $W'$, which is equivalent to saying there are matrices $X,Y,Z\in\GL_3(\ff)$ such that $(X,Y,Z)$ is an $\ff$-isotopism $t\rightarrow t'$.
In particular, $\LA_{E,P}(\ff)$ and $\LA_{E',P'}(\ff)$ are isomorphic if and only if $\ps{\J{E'}{P'}}=\ps{\J{E}{P}}$; cf.\ \cref{def:ps}. 

Now let $\cor{L}$ and $\cor{L}'$ be line bundles on $C=E$ and $C'=E'$ and let $\gamma:E\rightarrow E'$ be linear as given by \cref{prop:M-bundles}. Without loss of generality, we take $\cor{L}=\cor{L}(P-\cor{O})$ and $\cor{L'}=\cor{L}(P'-\cor{O}')$; cf.\ \eqref{eq:D} and \eqref{eq:pic0}. Moreover, note that the line bundles $\cor{M}=\cor{L}(3\mathcal{O})$ and $\cor{M}'=\cor{L}(3\cor{O}')$ define the given embeddings of $E\rightarrow \P^2_{\ff}$ and $E'\rightarrow \P^2_{\ff}$; see also \cite[Prop.~III.3.1]{silverman}. In particular, the condition on the map $\gamma$ can be replaced with the existence of a (not necessarily linear) isomorphism $\delta:E\rightarrow E'$ such that 
\begin{equation}\label{eq:bdl1}
\delta^*\cor{L}'=\cor{L} \textup{ and } \delta^*\cor{M}'=\cor{M}.
\end{equation}
We now show that the existence of $\delta$ satisfying \eqref{eq:bdl1} is equivalent to the existence of an isomorphism of elliptic curves $\varphi: E\rightarrow E'$ with the property that $\varphi(P)=P'$.
By using the symbol $\delta$ also for the map $\Div(E)\rightarrow\Div(E')$ that is induced by $\delta$, we rewrite \eqref{eq:bdl1} as
\[
\delta^*\cor{L}'=\cor{L} \textup{ and } \cor{L}(3\delta^{-1}(\cor{O}'))=\delta^*\cor{L}(3\cor{O'})=\delta^*\cor{M}'=\cor{M}=\cor{L}(3\cor{O}).
\]
We derive from \cref{lem:divlinbdl} that $[3\delta^{-1}(\cor{O}')]=[3\cor{O}]$, in other words 
$
[3(\delta^{-1}(\cor{O}')-3\cor{O})]= 0.
$
It now follows from Abel's Theorem, i.e.\ \cref{lem:abel}, that $[3](\delta^{-1}(\cor{O}')\ominus\cor{O})=\cor{O}$ and so $\delta^{-1}(\cor{O}')$ belongs to $E[3](\ff)$. Let $\tau$ denote translation by $\delta^{-1}(\cor{O}')$ on $E$. We then get that $\delta\circ\tau:E\rightarrow E'$ is an isomorphism of elliptic curves. Set $\varphi=\delta\circ\tau$. To conclude the proof, we show that $\varphi(P)=P'$. For this, note that $\varphi^*\cor{L}=\varphi^{*}\cor{L}(P-\cor{O})=\cor{L}(\varphi^{-1}(P)-\varphi^{-1}(\cor{O}))$. As a consequence, \eqref{eq:bdl1} and \cref{lem:divlinbdl} yield
\[
[\varphi^{-1}(P')-\cor{O}]=[\varphi^{-1}(P')-\varphi^{-1}(\cor{O}')]=[ P-\cor{O}].
\]
Abel's Theorem implies now that $\varphi^{-1}(P')=P$ equivalently that $\varphi(P)=P'$.
\end{proof}

\begin{proof}[Proof of \Cref{mainthm:sat}]
    Let $t:U\times W\rightarrow T$ and $t':U'\times W'\rightarrow T'$ be the
    $\ff$-tensors defined by
    $\J{E}{P}$ and $\J{E'}{P'}$, respectively. By
    \cref{thm:Baer-correspondence},  the groups $\GP_{E, P}(\ff)$ and $\GP_{E', P'}(\ff)$ are isomorphic if
    and only if there exists an $\F_p$-pseudo isometry between $t$ and
    $t'$. From \cref{cor:adj-cent-elliptic}, we know that $\Cent(t) \cong \Cent(t') \cong
    \ff$. Thus,  arguing as in the proof of \cref{prop:reduction}, we get that 
    $\pseudo_{\F_p}(t, t')$ can be considered as contained in the set $\pseudo_{\ff}(t, t')\times
    \Gal(\ff/\F_p)$. Therefore, $t$ and $t'$ are
    $\F_p$-pseudo isometric if and only if there exists $\sigma \in
    \Gal(\ff/\F_p)$ such that ${}^{\sigma}t$ and $t'$ are $\ff$-pseudo
    isometric. The matrix of linear forms associated to ${}^{\sigma}t$ is
    $\tuB_{{\sigma}(E), {\sigma}(P)}$. Again by
    \cref{thm:Baer-correspondence},  the tensors ${}^{\sigma}t$ and $t'$ are $\ff$-pseudo
    isometric if and only if $\LA_{{\sigma}(E), {\sigma}(P)}(\ff)$ and $\LA_{E',
    P'}(\ff)$ are isomorphic $\ff$-Lie algebras.  To conclude apply \cref{th:awesome}.
\end{proof}

\begin{remark}[The work of Ng]\label{rmk:Ng}
For a nice overview of the geometry of the $\Gamma$-orbits of $\mathcal{L}_{\mathcal{C}}$ we refer to \cite[Sec.~2]{Ng}. 
\cref{th:awesome} is a specialized variation of  Theorem~1 of~\cite{Ng}, which classifies the complex $\Gamma$-orbits in terms of triples $(E,L_1,L_2)$ where $L_1$ and $L_2$ are particularly chosen line bundles on $E$; cf.\ \eqref{eq:bdl1}.  The study in \cite{Ng} goes beyond the smooth case classifying also singular cuboids up to $\Gamma$-equivalence. In future work, we hope to come back to the investigation of the class of groups arising from this last family. 
\end{remark}

\subsection{Automorphisms of elliptic groups and the proof of
\Cref{mainthm:aut}}\label{sec:proof-aut}

In this section we compute the size of the automorphism group of a group arising
from a tensor satisfying $\cor{T}(t)\geq 2$ and whose Pfaffian class defines a
cubic with a flex point over $\ff$. Indeed, using pseudo-isometries, any such
curve can be put in short Weierstrass equation, and thus the group in question
will be isomorphic to a group of the form $\GP_{E,P}(\ff)$. 

\begin{thm}\label{thm:pseudo}
    Let $E$ be an elliptic curve in $\P^2_{\ff}$ given by a Weierstrass equation,
    and let $P\in E(\ff)\setminus\{\cor{O}\}$. Assume that $\charac(\ff) =
    p\geq 5$. Then 
    \[ 
        |\pseudo_{\ff}(\tuB_{E,P})| = \frac{|\Aut_{\cor{O}}(E)|}{|\Aut_{\cor{O}}(E)\cdot P|}\cdot|E[3](\ff)|\cdot\begin{cases}
        |\GL_2(\ff)| & \textup{if } P\in E[2](\ff),\\
        2(q-1)^2 & \textup{otherwise}. 
        \end{cases}
    \]
\end{thm}

\begin{proof}
    Let $t: V\times V \to T$ be the tensor defined by $\tuB_{E,P}\in
    \Mat_6(\ff[y_1,y_2,y_3])$ and let $\LA
    = \LA_t(\ff)$ be the Lie algebra of $\GP_{t}(\ff)$ via the Baer
    correspondence. 
    We compute the order of
    $\pseudo_{\ff}(t)$. Following the strategy in
    \cite{stanojkovski2019hessian}, we work with the automorphism group
    $\Aut(\LA)=\Aut_{\ff}(\LA)$ of the $\ff$-algebra $\LA$. For this, let $U,W$
    be distinct elements of $\ti{t}$ corresponding to the base choice
    yielding $\tuB_{E,P}$ and note that $V=U\oplus W$. We define the following
    subgroups of $\Aut(\LA)$:
    \begin{itemize}
    \item $\Aut_V(\LA)=\{\alpha\in\Aut(\LA) \mid \alpha(V)=V \}$ and
    \item $\Aut_V^f(\LA)=\{\alpha\in\Aut(\LA) \mid \alpha(U)=U,\, \alpha(W)=W
    \}$.
    \end{itemize}
    Since $t$ is full, if $(\alpha,\beta) \in \pseudo_{\ff}(t)$, then
    $\alpha$ uniquely determines $\beta$. 
    In particular, we have that $\pseudo_{\ff}(t)\cong \Aut_V(\LA)$.
    We now look at the action of $\Aut_V(\LA)$ on $\ti{t}$. Using
    \cref{lem:swap,lem:psi-M}, we derive that this action is $2$-transitive.
    It follows that the stabilizer of the pair $(U,W)$ is equal to
    $\Aut_V^f(\LA)$ and has index $|\cor{T}(t)|(|\cor{T}(t)|-1)$ in
    $\Aut_V(\LA)$. By \cref{cor:four-cases}, it thus holds that
    \begin{equation}
        |\pseudo_{\ff}(\tuB_{E,P})| = |\Aut_V(\LA)|=|\Aut_V^f(\LA)|\cdot\begin{cases}
        q(q+1) & \textup{if } P\in E[2](\ff),\\
        2 & \textup{otherwise}. 
    \end{cases}
    \end{equation}
    To determine $|\Aut_V^f(\LA)|$ 
we note that, via \cref{rmk:autotopism-pseudo},  an element
$$\diag(X,Y,Z)=\begin{pmatrix}
X & 0 & 0 \\ 0 & Y & 0 \\ 0 & 0 & Z
\end{pmatrix} \in\GL_9(\ff),  \textup{ with } X,Y,Z\in\GL_3(\ff)$$
belongs to $\Aut_V^f(\LA)$ if and only if
    $X^{\tp}\J{E}{P}(Z{\bm y})Y=\J{E}{P}({\bm y})$.  Since the change of coordinates given by $Z$ maps $E$ to itself,  the following map is well defined
    \[ 
    \varphi:
    \Aut_V^f(\LA)\rightarrow \Aut(E), \quad    \diag(X, Y, Z) \mapsto Z.
    \]
    Thus, $\varphi$ maps into the linear part of $\Aut(E)$, namely those
    automorphisms of $E$ that extend to linear transformations of $\P^2_{\ff}$. From the proof of \cref{th:awesome}, we know that
    \[
        |\im\varphi| = |E[3](\ff)|\cdot|\{\varphi\in\Aut_{\cor{O}}(E) : \varphi(P)=P\}|=|E[3](\ff)|\cdot\frac{|\Aut_{\cor{O}}(E)|}{|\Aut_{\cor{O}}(E)\cdot P|}.
    \]
    We claim that $|\ker\varphi| = (q-1)^2$. To prove this, we start by
    observing that 
    \[
    \ker\phi=\{\diag(X,Y,cI_3) \mid X,Y\in\GL_3(\ff),\, c\in\ff^\times, X^{\tp}\J{E}{P}(c{\bm y})Y=\J{E}{P}({\bm y})\}.
    \]
    The last equality in the definition of $\ker\phi$ can be rewritten as
    $(cX^{\tp})\J{E}{P}Y=\J{E}{P}$, so since $Y$ is invertible, we get a map
    \[
    \ker\varphi\longrightarrow\Adj(\J{E}{P}), \quad \diag(X,Y,cI_3)\longmapsto(cX,Y^{-1}).
    \]
    All elements of the form $(aI_3,bI_3,(ab)^{-1}I_3)$ with $a,b\in\ff^\times$
    are elements of $\ker\phi$, and as a consequence of \cref{prop:JP-adj}, the
    converse is also true. Indeed, if $\diag(X,Y,cI_3)$ belongs to
    $\ker\varphi$, then $(cX,Y^{-1})$ 
    is an element of
    $\Adj(\J{E}{P})=\{(kI_3,kI_3) \mid k\in \ff\}$. In particular, there exists
    $k\in\ff^\times$ such that $cX=kI_3=Y^{-1}$. It then follows that
    $X=c^{-1}kI_3$ and $Y=k^{-1}I_3$. Thus, 
    $\ker\varphi=\{(aI_3,bI_3,(ab)^{-1}I_3)\mid a,b\in\ff^\times\}$, and we conclude that $|\ker\varphi| = |\ff^{\times}|^2=(q-1)^2$.  Since $|\GL_2(\ff)|=q(q+1)(q-1)^2$ the proof is complete.
\end{proof}

For an elliptic curve $E$ in $\P^2_{\ff}$ given by a short Weierstrass equation and $P\in E(\ff)\setminus\{\cor{O}\}$, we write 
\[
\Gal_{E,P}(\ff/\F_p)=\{\sigma\in\Gal(\ff/\F_p) \, : \, \pseudo(\tuB_{E,P},\tuB_{\sigma(E),\sigma(P)})\neq \varnothing\}.
\]
The following two corollaries follow in a straightforward way from \cref{prop:reduction} and \cref{thm:pseudo}.  Note that the denominators on the left side are $|\Hom_{\ff}(V,T)|$ and $|\Hom_{\F_p}(V,T)|$, respectively.

\begin{coro}
    Let $E$ be an elliptic curve in $\P^2_{\ff}$ given by a Weierstrass equation.
    Moreover, let $P\in E(\ff)\setminus\{\cor{O}\}$. Assume that
    $\charac(\ff)=p\geq 5$. Then the following holds:
    \[
    \frac{|\Aut_{\ff}(\LA_{E,P}(\ff))|}{q^{18}}=|E[3](\ff)|\cdot \frac{|\Aut_{\cor{O}}(E)|}{|\Aut_{\cor{O}}(E)\cdot P|}\cdot\begin{cases}
    |\GL_2(\ff)| & \textup{if } P\in E[2](\ff),\\
    2(q-1)^2 & \textup{otherwise}. 
    \end{cases}
    \]
\end{coro}

\begin{coro}\label{cor:pseudo}
    Let $E$ be an elliptic curve in $\P^2_{\ff}$ given by a Weierstrass equation and $P\in E(\ff)\setminus\{\cor{O}\}$. Assume that
    $\charac(\ff)=p\geq 5$ and $|\ff|=p^e$. Then the following holds:
    \[
    \frac{|\Aut(\GP_{E,P}(\ff))|}{p^{18e^2}}=|\Gal_{E,P}(\ff/\F_p)|\cdot|E[3](\ff)|\cdot \frac{|\Aut_{\cor{O}}(E)|}{|\Aut_{\cor{O}}(E)\cdot P|}\cdot\begin{cases}
    |\GL_2(\ff)| & \textup{if } P\in E[2](\ff),\\
    2(p^e-1)^2 & \textup{otherwise}. 
    \end{cases}
    \]
\end{coro}

\begin{remark}\label{rmk:computing-galois}
Let $E$ be an elliptic curve given by the Weierstrass equation \eqref{eq:E-sWe} over $\ff$ and let $P=(\lambda,\mu,1)\in E(\ff)$.
To compute the size of $\Gal_{E,P}(\ff/\F_p)$ one can rely on \cref{th:awesome} in the following way.  For $\sigma$ to belong to $\Gal_{E,P}(\ff/\F_p)$ a necessary and sufficient condition is that $\pseudo_F(\tuB_{E,P},\tuB_{\sigma(E),\sigma(P)})$ be non-empty. Thanks to  \cref{th:awesome} the last condition is equivalent to the existence of an isomorphism of elliptic curves $\varphi:E\rightarrow \sigma(E)$ such that $\varphi(P)={\sigma}(P)$. With the aid of, for instance, \cite[Tab.~III.3.1, p.~45]{silverman}, one shows that such a $\varphi$ is given by a map $(x,y)\mapsto (u^2x,u^3y)$ where $u\in\ff$ is chosen such that  $(\sigma(\lambda),\sigma(\mu))=(u^2\lambda,u^3\mu)$ and 
\[
(u^4,u^6)=\begin{cases}
(u^4,\sigma(b)b^{-1}) & \textup{ if } a=0 \quad (\textup{equiv. } j(E)=0), \\
(\sigma(a)a^{-1},u^6) & \textup{ if } b=0 \quad (\textup{equiv. } j(E)=1728), \\
(\sigma(a)a^{-1},\sigma(b)b^{-1}) & \textup{ otherwise.}
\end{cases}
\]
\end{remark}

\begin{ex}
Let $E$ be the elliptic curve defined over $\F_5$ by $y^2=x^3-2x$ and note that $E[2](\F_5)=\{\cor{O},(0,0)\}$ while, setting $\ff=\F_5[\sqrt{2}]$, we get 
\[E[2](\ff)=\{\cor{O},(0,0),(\sqrt{2},0),(-\sqrt{2},0)\}.\]
The matrices $\J{E}{P}$ corresponding to $P\in E[2](\ff)\setminus\{\cor{O}\}$ are equivalent to the Hessian matrices given in \cite[Sec.~1.4]{stanojkovski2019hessian} where $\delta$ is chosen to be $2$.  We show that for each choice of $P$, the group $\Gal_{E,P}(\ff/\F_p)$ coincides with $\Gal(\ff/\F_p)$. For this, we note that in this case $b=0$, so we can identify $\Aut_{\cor{O}}(E)$ with $\F_5^\times$; cf.\ \cite[Sec.~III.10]{silverman}.  Following the notation from \cref{rmk:computing-galois}, we fill the following table:
\[
\begin{array}{c|c|c|c|c}
\sigma & u^4 & \sigma(0,0) & \sigma(\sqrt{2},0) & \sigma(-\sqrt{2},0)\\
\hline
\id & 1 & (0,0) & (\sqrt{2},0) & (-\sqrt{2},0) \\
x\mapsto x^5 & 1 & (0,0) & (-\sqrt{2},0) & (\sqrt{2},0)
\end{array}
\]
Taking $u=1$ in the first row and $u=2$ in the second yields the claim.
\end{ex}

\begin{ex}\label{ex:Galois-required}
    Let $f(x) = x^2 - x + 2 \in \F_5[x]$, which is irreducible. Set
    $\ff = \F_5[x]/(f(x))$, and let $\alpha\in \ff$ be a root of $f$.
    Let $E$ be the elliptic curve in $\P^2_{\ff}$ given by $y^2 = x^3 +
    \alpha x + \alpha$. The $j$-invariant of $E$ is $\alpha - 1$. Let $\sigma\in
    \Gal(\ff/\F_5)$ be the map $x\mapsto x^5$ so that  ${\sigma}(E)$  is defined
    by $y^2 = x^3 + \alpha^5 x + \alpha^5$. The $j$-invariant of $\sigma(E)$ is
    $\alpha^5-1 = -\alpha$, so the elliptic curves $E$ and ${\sigma}(E)$ are not
    isomorphic. Let now $P = (\alpha^3, \alpha)\in E(\ff)$ and note that
    $\sigma(P) =(\alpha^{15},\alpha^5)\in {\sigma}(E)(\ff)$. It follows that
    $\LA= \LA_{E, P}(\ff)$ and $\LA_\sigma= \LA_{\sigma(E), \sigma(P)}(\ff)$ are
    not isomorphic as $\ff$-Lie algebras -- that is, there is no $\ff$-linear
    isomorphism $\LA\to \LA_\sigma$ since $E$ and ${\sigma}(E)$ are not
    isomorphic. However,  $(I_9, \sigma) \in \GL_9(\ff)\rtimes
    \Gal(\ff/\F_5)$ yields an $\ff$-semilinear isomorphism
    $\LA\rightarrow\LA_{\sigma}$. Thus, as $\F_5$-Lie algebras,
    $\LA\cong\LA_\sigma$ and, consequently, $\GP_{E, P}(\ff)\cong
    \GP_{\sigma(E), \sigma(P)}(\ff)$.
To make this isomorphism explicit, observe that
    \begin{align}\label{eqn:J_EP(K)}
        J & = \J{E}{P}= 
        \begin{pmatrix} 
            x - \alpha^3z & y - \alpha z & 0 \\ 
            y + \alpha z & \alpha^3 x + (\alpha + \alpha^6)z & x \\ 
            0 & x & -z 
        \end{pmatrix}.
    \end{align} 
    Viewing $\ff$ as a $2$-dimensional vector space with basis
    $\{1, \alpha\}$ over $\F_5$, we rewrite $J$ from~\eqref{eqn:J_EP(K)} as
    an $\F_5$-tensor $\F_5^6\times \F_5^6
    \to \F_5^6$. The associated matrix of linear forms is 
    \begin{align*} 
        \overline{J} &= {\tiny \begin{pmatrix} 
            x_1 + 2z_1 + z_2 & x_2 + 3z_1 + 3z_2 & y_1 - z_2 & y_2 + 2z_1 - z_2 & 0 & 0 \\ 
            x_2 + 3z_1 + 3z_2 & 3x_1 + x_2 - z_1 + z_2 & y_2 + 2z_1 - z_2 & 3y_1 + y_2 + 2z_1 + z_2 & 0 & 0 \\
            y_1 + z_2 & y_2 + 3z_1 + z_2 & 3x_1 - x_2 + 2z_1 + z_2 & 2x_1 + 2x_2 + 3z_1 + 3z_2 & x_1 & x_2 \\
            y_2 + 3z_1 + z_2 & 3y_1 + y_2 + 3z_1 - z_2 & 2x_1 + 2x_2 + 3z_1 + 3z_2 & x_1 - x_2 - z_1 + z_2 & x_2 & 3x_1 + x_2 \\
            0 & 0 & x_1 & x_2 & - z_1 & -z_2 \\
            0 & 0 & x_2 & 3x_1 + x_2 & -z_2 & 2z_1 - z_2 
        \end{pmatrix}} .
    \end{align*} 
    We do the same construction for the matrix ${\sigma}(J)$ associated to ${\sigma}(E)$ and ${\sigma}(P)$: 
    \begin{align*} 
        \overline{{\sigma}(J)} &= {\tiny \begin{pmatrix} 
            x_1 + 3z_1 - z_2 & x_2 + 2z_1 + 2z_2 & y_1 - z_1 + z_2 & y_2 + 3z_1 & 0 & 0 \\ 
            x_2 + 2z_1 + 2z_2 & 3x_1 + x_2 + z_1 - z_2 & y_2 + 3z_1 & 3y_1 + y_2 + 3z_2 & 0 & 0 \\
            y_1 + z_1 - z_2 & y_2 + 2z_1 & 2x_1 + x_2 + 3z_1 - z_2 & 3x_1 + 3x_2 + 2z_1 + 2z_2 & x_1 & x_2 \\
            y_2 + 2z_1 & 3y_1 + y_2 + 2z_2 & 3x_1 + 3x_2 + 2z_1 + 2z_2 & -x_1 + x_2 + z_1 - z_2 & x_2 & 3x_1 + x_2 \\
            0 & 0 & x_1 & x_2 & - z_1 & -z_2 \\
            0 & 0 & x_2 & 3x_1 + x_2 & -z_2 & 2z_1 - z_2 
        \end{pmatrix}} .
    \end{align*} 
To define the isomorphism corresponding to $(I_9, \sigma)$,
    we define a $6\times 6$ block diagonal matrix $D = \diag(X, X, X)$ by setting
    \begin{align*} 
        X &= \begin{pmatrix} 
            1 & 1 \\ 0 & -1 
        \end{pmatrix} .
    \end{align*} 
    Note that $X^{-1} = X$. One can check that $D^{\tp}\overline{J}(D\bm{x})D=
    \overline{{\sigma}(J)}$, implying that 
    \begin{align*} 
        \begin{pmatrix} 
            D^{\tp} & 0 \\ 0 & D^{\tp} 
        \end{pmatrix} \begin{pmatrix} 
            0 & \overline{J}(D\bm{x}) \\ 
            -\overline{J}^{\tp}(D\bm{x}) & 0 
        \end{pmatrix} \begin{pmatrix} 
            D & 0 \\ 0 & D
        \end{pmatrix} 
        &= \begin{pmatrix} 
            0 & \overline{{\sigma}(J)} \\ 
            -\overline{{\sigma}(J)}^{\tp} & 0 
        \end{pmatrix} .
    \end{align*} 
    Thus, $(\diag(D, D), D)$ is an $\F_5$-pseudo isometry, and
    therefore, $\GP_{E, P}(\ff)\cong \GP_{\sigma(E), \sigma(P)}(\ff)$. This
    argument also shows that the Galois part of $\Aut(\GP_{E, P}(\ff))$ must be
    trivial -- otherwise one would for example get that  $E$ and ${\sigma}(E)$ are isomorphic. 
  \cref{prop:reduction} yields
    \begin{align*} 
        \Aut(\GP_{E, P}(\ff)) \cong \Hom_{\F_5}(\F_5^{12}, \F_5^6) \rtimes \pseudo_{\ff}(\tuB_{E,P}). 
    \end{align*} 
\end{ex}

\section{Isomorphism testing for $E$-groups}\label{sec:algs}

The main goal of this section is to prove \Cref{mainthm:iso}, which we do in
\cref{sec:iso-proof}. For this, we develop a number of algorithms to recognize
when a $p$-group is (isomorphic to) an $E$-group. We even consider the more
general situation computing an isotropic decomposition of alternating tensors
with irreducible Pfaffian in \cref{sec:recognition}. We end with a discussion in
\cref{sec:implementation} about our implementation of \cref{mainthm:iso} in
\textsf{Magma}~\cite{magma/97}, where we also provide some examples.

\subsection{Computational models for finite groups} 
\label{sec:computational-model}

Common computational models for finite groups are given by (small) sets of
generators in either (1) matrix groups over finite fields~\cite{Luks/92} or (2)
permutation groups~\cite{Seress/03}. Although polycyclic and
``power-commutator'' presentations seem to work well in practice, it is not
known whether multiplication with these presentations can be done in polynomial
time~\cite{Leedham-GreenSoicher/98}. 

Throughout, we work with the ``permutation group quotient'' model proposed by
Kantor and Luks in~\cite{KantorLuks/90}. This avoids issues where the given
$p$-group $G$ can only be faithfully represented as a permutation group on a set
whose size is approximately $|G|$; see Neumann's example in~\cite{Neumann/86}.
The following proposition ensures that we can efficiently go from groups to
tensors using the permutation group quotient model. With appropriate bounds for
the prime $p$, one can achieve this for the matrix group model as
well~\cite{Luks/92}. 

\begin{prop}[{\cite[Sec.~4]{KantorLuks/90}}]\label{prop:basic}
    Given a group $G$ as a quotient of a permutation group, 
    each of the following problems has an algorithm that runs in time polynomial in $\log|G|$.
    \begin{enumerate}[label=$(\arabic*)$] 
        \item Compute $|G|$.
        \item Given $x, x_1, \dots, x_k\in G$ either write $x$ as a word in $\{x_1,\dots, x_k\}$, or else prove $x\notin \langle x_1,\dots, x_k\rangle$. 
        \item Find generators for $\ZG(G)$ and for $[G,G]$.
        \item Decide if $G$ is nilpotent of class $2$.
        \item If $G$ is nilpotent of class $2$, construct the commutator tensor of $G$ as a matrix of linear forms.
    \end{enumerate} 
\end{prop}

\subsection{Isotropic decompositions}
\label{sec:recognition}

We depart from the focused setting of $E$-groups, and we consider the more
general setting of computing an isotropic decomposition for an alternating
$\ff$-tensor $t : V\times V\to T$.

\begin{thm}\label{thm:recognition}
    Let $\tuB\in \Mat_{2n}(\ff[y_1,\dots, y_d]_1)$ be skew-symmetric. Then there
    exist Las Vegas algorithms, using $O(n^6d + n^2d^5)$ and $O(n^6d)$
    operations in $\ff$ respectively, to do the following.
    \begin{enumerate}[label=$(\arabic*)$]
        \item\label{thm-part:cent} Write $\tuB$ over $C=\Cent(\tuB)$, provided
        $C$ is a field, call it $\tuB_{C}$. 
        \item\label{thm-part:ti-sub} Return $X\in \GL_{2n}(\ff)$ such that
        $X^{\tp}\tuB_{C} X$ gives an isotropic decomposition whenever
        $\tuB_{C}$ is decomposable with irreducible Pfaffian.
    \end{enumerate} 
\end{thm}

We prove \Cref{thm:recognition} with the next two lemmas. Common to both
algorithms, and also their bottleneck, is computing a basis of a large system of
linear equations. 

\begin{lem}\label{lem:cent-rewrite}
    There exists a Las Vegas algorithm that, given $\tuB\in\Mat_{m\times
    n}(\ff[y_1,\dots, y_d]_1)$, decides whether its centroid is a field extension $C/\ff$ and, if so, 
    returns the matrix of linear forms $\tuB_C$ over $C$ using $O(mnd(m^2 + n^2
    + d^2)^2)$ operations in $\ff$. 
\end{lem}

\begin{proof} 
    \emph{Algorithm.} Compute a basis for the centroid $C=\Cent(\tuB)$  by solving a system of linear equations.
 Use \cite[Th.~1.3]{BrooksbankWilson/15} to determine if $C$ is a field or not, and if $C$ is a field, then find a
    generator $X=(X_1,X_2,X_3)$ of $C$ as a unital $\ff$-algebra. If $C$ is not a field, then just return
    $\tuB$, so we assume $C$ is a field. Let $V_1=\ff^m$, $V_2=\ff^n$, and
    $V_3=\ff^d$. 
    Find a
    maximal set of nonzero vectors $\mathcal{B}_i$ in $V_i$ which are all in
    different $X_i$ orbits. For each $(u,v) \in \mathcal{B}_1 \times
    \mathcal{B}_2$, write $u^{\tp}\tuB v$ as $\sum_{w\in\mathcal{B}_3}
    f_{u,v}^{(w)}(X)w$, where $f_{u,v}^{(w)}(z)\in \ff[z]$. Return the
    matrix $\tuB_C=(f_{u,v}^{(w)}(X))_{u,v,w}$, running through all
    $(u,v,w)\in\mathcal{B}_1\times \mathcal{B}_2\times \mathcal{B}_3$.

    \emph{Correctness.} We only ever proceed beyond the construction of a basis
    for $C$ if $C$ is a field. In this case, the $V_i$ are $C$-vector spaces.
    The sets $\mathcal{B}_i$ are $C$-bases for the $V_i$. 
    
    \emph{Complexity.} A basis for the centroid $C$ is computed by solving a
    homogeneous system of $mnd$ linear equations in $m^2+n^2+d^2$ variables over
    $\ff$. By \cite[Th.~1.3]{BrooksbankWilson/15}, determining whether or
    not $C$ is a field is done constructively in polynomial time using a Las
    Vegas algorithm, but this complexity is dominated by the complexity to
    compute a basis for $C$. The complexity to find bases $\mathcal{B}_i$ and
    rewrite $\tuB$ over $C$ is also dominated by the complexity for the
    centroid. 
\end{proof} 

\begin{lem}\label{lem:ti-sub-alg}
    There is an algorithm that, given $\tuB \in
    \Mat_{2n}(\ff[y_1,\dots, y_d]_1)$ skew-symmetric, returns $X\in\GL_{2n}(\ff)$ such that
    $X^{\tp}\tuB X$ gives an isotropic decomposition whenever $\tuB$ is
    decomposable with irreducible Pfaffian, using $O(n^{6}d)$ operations in
    $\ff$.
\end{lem}

\begin{proof} 
    \emph{Algorithm.} Compute $A=\Adj(\tuB)$. Determine if $A$ is isomorphic to
    either $\mathbf{X}_1(\ff)$ or $\mathbf{S}_2(\ff)$. If not, report that
    $\tuB$ is not isotropically decomposable with an irreducible Pfaffian.
    Otherwise, construct an isomorphism of $*$-algebras $\phi: A \rightarrow S$,
    where $S$ is either $\mathbf{X}_1(\ff)$ or $\mathbf{S}_2(\ff)$. Set 
    \begin{align*}
        \mathcal{E} = \begin{cases} 
            \{E_{11}, E_{22}\} & \text{if } S = \mathbf{S}_2(\ff), \\
            \{(1, 0), (0, 1)\} & \text{if } S = \mathbf{X}_1(\ff),
        \end{cases}
    \end{align*} 
    where $E_{ij}$ is the matrix with $1$ in the $(i,j)$ entry and $0$
    elsewhere. For each $e \in \mathcal{E}$, set $(Y, Z) = \phi^{-1}(e)\in A$ and let $U_e\leq \ff^n$ be the column
    span of $Y$. Let $\mathcal{B}$ a basis
    for $\ff^{2n}$ containing bases for $U_e$ for all $e\in \mathcal{E}$, and
    return the transition matrix $X\in\GL_{2n}(\ff)$.

    \emph{Correctness.} By Theorem~\ref{thm:reducible-adjoints}, if $\tuB$ is
    isotropically decomposable with an irreducible Pfaffian, then $A$ is
    isomorphic to either $\mathbf{X}_1(\ff)$ or $\mathbf{S}_2(\ff)$. In both
    cases, $\mathcal{E}$ is a set of elements in $S$ whose images induce
    distinct, $n$-dimensional, totally-isotropic subspaces;  cf.\ the proof of \cref{prop:adj-X-S}. To show that $U_e$
    is totally-isotropic,  set $(Y, Z)=\phi^{-1}(e)$. For each $u,v\in \ff^{2n}$ one has $u^{\tp}Y^{\tp}\tuB
    Yv = u^{\tp}\tuB ZYv = 0$,  proving that $U_e$ is totally isotropic.

    \emph{Complexity.} A basis for $\Adj(\tuB)$ is constructed by solving $n^2d$
    (homogeneous) linear equations in $2n^2$ variables.
    By~\cite[Th.~4.1]{BrooksbankWilson/12}, the complexity for the
    constructive recognition of $*$-algebras is dominated by the cost of
    constructing a basis for $\Adj(\tuB)$. 
\end{proof}

\begin{proof}[Proof of \Cref{thm:recognition}]
    Use \Cref{lem:cent-rewrite} for \ref{thm-part:cent}, and for
    \ref{thm-part:ti-sub}, apply \Cref{lem:ti-sub-alg}.
\end{proof}

\subsection{Isomorphism testing of $E$-groups and the proof of \Cref{mainthm:iso}} 
\label{sec:iso-proof}

\cref{thm:recognition} is almost enough to prove the first two statements of
\cref{mainthm:iso}. The next lemma fills the gap by providing an algorithm to
find flex points on smooth cubics; cf.\ \cref{rmk:flex}.

\begin{lem}\label{lem:flex-computation}
    Given a homogeneous cubic $f\in \ff[y_1,y_2,y_3]$, there exists an algorithm
    that decides if $f$ is smooth and if so returns all $a\in
    \ff^3\setminus\{(0,0,0)\}$ such that $f(a)=0$ and $a$ is a flex point of
    $f$, which uses $O(\log q)$ field operations. 
\end{lem} 

\begin{proof}
    \emph{Algorithm.} Compute a Gr\"obner basis $\mathcal{G}_1$ for the ideal
    \[ 
        I_1 = \langle \partial f / \partial y_1, \partial f / \partial y_2,
        \partial f / \partial y_3, f\rangle \textup{ of } \ff[y_1,y_2,y_3]
    \] 
   using the lexicographical monomial order. Let $g\in\mathcal{G}_1$ be
    homogeneous with at most $2$ variables. Factor $g$ using univariate
    algorithms~\cite[Ch.~14]{vzGG/13}, and use those solutions to find all the
    solutions to the polynomial system determined by $\mathcal{G}_1$. If a
    solution exists, declare $f$ ``singular,'' and return a singular point.
    Otherwise $f$ is smooth.
    
    Compute the determinant of the Hessian matrix of $f$, which yields a
    homogeneous cubic $H\in \ff[y_1,y_2,y_3]$. Construct a Gr\"obner basis
    $\mathcal{G}_2$ for the ideal $I_2 = \langle f, H\rangle$ using the
    lexicographical monomial order. Return the set of solutions to the
    polynomial system determined by $\mathcal{G}_2$ in the same fashion as above.
    
    \emph{Correctness.} The algorithm to decide whether $f$ is smooth is correct
    by definition and the fact that $\langle \mathcal{G}_1\rangle = I_1$. The
    existence of such a $g\in \mathcal{G}_1$ is the content of the Elimination
    Theorem~\cite[Sec.~3.1]{CLO}. If the algorithm starts to compute the flexes,
    we know that $f$ is, therefore, smooth. By B\'ezout's theorem the number of
    flex points is bounded above by $9$. Thus, $I_2$ is $0$-dimensional and the
    flexes are the solutions to the polynomial system determined by
    $\mathcal{G}_2$. 
    
    \emph{Complexity.} Computing a Gr\"obner basis for a set of polynomials
    whose degree and number of variables are constant is done using $O(1)$ field
    operations. We can solve the polynomial systems determined by
    $\mathcal{G}_i$ by calling a constant number of univariate factoring
    algorithms. Since the degrees are bounded by some absolute constant,
    factoring is done using $O(\log q)$ field operations~\cite[Ch.~14]{vzGG/13}. 
\end{proof}

\begin{proof}[Proof of \Cref{mainthm:iso}\ref{thm-part:recog1}]
    By \cref{cor:adj-cent-elliptic}, if the $G_i$ are elliptic groups, the
    $\Cent(t_{G_i})$ are finite extensions of $\F_p$. Write $t_{G_i}$ as a
    matrix of linear forms $\widetilde{\tuB}_i \in
    \Mat_{6m}(\mathbb{F}_p[y_1,\dots, y_{3m}]_1)$.  Use
    \Cref{thm:recognition}\ref{thm-part:cent} to express $\widetilde{\tuB}_i$ as
    a matrix of linear forms $\tuB_i \in \Mat_6(\ff[y_1,y_2,y_3]_1)$ over the
    centroid $\ff$ of $\widetilde{\tuB}_i$, using $O(m^7)$ field operations. If
    the algorithm fails at any stage, then $G_i$ is not an elliptic group.
    Otherwise compute the Pfaffians of $\tuB_i$. Then apply
    \Cref{lem:flex-computation} to decide if the Pfaffians are elliptic curves
    containing flex points; this uses $O(\log |\ff|) = O(m\log p)$ field
    operations. If they are not, then $G_i$ is not isomorphic to some
    $\GP_{E,P}(\ff)$. 
\end{proof}

Now our objective is to develop the algorithms that feed into the algorithm for
\cref{thm:computational-heart}, which is the main algorithm for
\cref{mainthm:iso}\ref{thm-part:iso}. We first describe some algorithms, which
will be used in \cref{thm:computational-heart}, that use a constant number of
field operations. 

\begin{lem}\label{lem:Aut-E}
    Assume $\charac(\ff)\geq 5$. Given an elliptic curve $E$ in short
    Weierstrass form, there is an algorithm returning the set of automorphisms
    $\Aut_{\cor{O}}(E)$ as a subgroup of $\GL_3(\ff)$ using $O(\log q)$ field
    operations. 
\end{lem} 

\begin{proof} 
    \emph{Algorithm.} Let $\mu_n(\ff)$ be the set of roots of $x^n - 1$ in
    $\ff$. Define $\mathcal{R}$ to be $\mu_2(\ff)$, $\mu_4(\ff)$, $\mu_6(\ff)$
    respectively when $j(E)\neq 0,1728$ or $j(E)=1728$ or $j(E)=0$. Return the
    subset $\{\diag(\omega^2, \omega^3, 1) ~|~ \omega \in \mathcal{R}\}$ of
    $\GL_3(\ff)$.

    \emph{Correctness.} It follows from \cite[Th.~III.10.1]{silverman}.

    \emph{Complexity.} Factoring constant-degree univariate polynomials is done
    using $O(\log q)$ field operations~\cite[Ch.~14]{vzGG/13}.
\end{proof}

The following is relevant in view of \cref{rmk:autotopism-pseudo}.

\begin{prop}\label{prop:autotopism}
    Assume $\charac(\ff) \geq 5$. Given an elliptic curve $E$ in short
    Weierstrass form and $P\in E(\ff)$, there is an algorithm returning a
    generating set for $\mathrm{Auto}_{\ff}(\J{E}{P})$ using $O(q^{1/4}\log q)$
    field operations.
\end{prop}

\begin{proof} 
    \emph{Algorithm.} Initialize $\mathcal{X}=\varnothing$. Use \cref{lem:Aut-E}
    to construct $\Aut_{\cor{O}}(E)$, and use \cref{lem:flex-computation} to get
    the set $R$ of flexes of $E$ over $\ff$. For each $\Omega\in
    \Aut_{\cor{O}}(E)$ fixing $P$ construct a nonzero $(X_{\omega},
    Y_{\omega})\in\Adj(\J{E}{P}, \J{E}{P}(\Omega\bm{y}))$ and include
    $(X_{\omega}, Y_{\omega}^{-1},\Omega)$ in $\mathcal{X}$. For each $Q\in R$,
    let $\tau_Q\in \GL_3(\ff)$ be the linear map given by translation by $Q$ on
    $E$, and construct a nonzero $(X_Q, Y_Q)\in \Adj(\J{E}{P}, \J{E}{P}(\tau_Q\bm{y}))$ and include $(X_Q, Y_Q^{-1}, \tau_Q)$ in $\mathcal{X}$.
    Find $a\in \ff$ such that $\langle a\rangle = \ff^{\times}$. Then return
    $\mathcal{X} \cup \{(aI_3, I_3, aI_3), (I_3, aI_3, aI_3)\}$.
    \

    \emph{Correctness.} As explained in \cref{rmk:flex}, the set $R$ is equal to
    $E[3](\ff)$.  As it is evident from \cref{rmk:autotopism-pseudo} and the proof of \Cref{thm:pseudo},     
    the map $\mathrm{Auto}_\ff(\J{E}{P})\to
    \Aut_V^f(\LA_{E, P}(\ff))$ given by $(X, Y, Z) \mapsto \diag(X, Y, Z)$ is an isomorphism, and thus correctness follows from the arguments there presented.

    \emph{Complexity.} The cardinalities of $\Aut_{\cor{O}}(E)$ and $E[3](\ff) =
    R$ are bounded from above by $6$ and $9$, respectively. The complexity is dominated by the complexity to find a
    primitive element of $\ff$. From the theorem of \cite{Shparlinski}, finding
    $a\in \ff^{\times}$ such that $\langle a\rangle = \ff^{\times}$ can be done
    in time $O(q^{1/4}\log q)$. 
\end{proof}

\begin{thm}\label{thm:computational-heart}
    Assume $\charac(\ff) = p \geq 5$. There exists an algorithm that, given a
    subfield $L\subset \ff$ and isotropically decomposable $\tuB,\tuB'\in \Mat_6(\ff[y_1, y_2,
    y_3]_1)$ whose Pfaffians define elliptic curves in $\P^2_{\ff}$ with
    flex points over $L$, returns the possibly empty coset
    $\pseudo_{L}(\tuB, \tuB')$ using $O(q)$ field operations. 
\end{thm}

\begin{proof} 
    \emph{Algorithm.} Use \cref{thm:recognition}\ref{thm-part:ti-sub} to rewrite
    $\tuB$ and $\tuB'$ such that they are isotropically decomposed. Let
    $\tuM,\tuM' \in \Mat_3(\ff[y_1,y_2,y_3]_1)$ be the top right $3\times 3$
    blocks in $\tuB$ and $\tuB'$, respectively. Construct $Z,Z'\in\GL_3(\ff)$
    with the property that both $f = \det(\tuM(Z\bm{y}))$ and $f' = \det(\tuM'(Z'\bm{y}))$
    yield short Weierstrass forms of the curves $E$ and $E'$. 

Set $\tuN = \tuM(Z\bm{y})$ and $\tuN' = \tuM'(Z'\bm{y})$ and find $P\in E(\ff)$ and $P'\in E'(\ff)$ such that
    $\Adj(\J{E}{P}, \tuN)$ and $\Adj(\J{E'}{P'}, \tuN')$ are nontrivial.
    Determine if there exists $\sigma\in \Gal(\ff/L)$  and an
    isomorphism $\phi: E'\to \sigma(E)$ of elliptic curves such that $P'\mapsto \sigma(P)$. If no
    such $\sigma$ exists, return $\varnothing$; otherwise, let $\varphi$ be represented by a matrix in $\GL_3(\ff)$ as given in \cref{lem:Aut-E} and choose
    \begin{itemize}
    \item $(X_1,Y_1)\in
    \Adj(\J{\sigma(E)}{\sigma(P)}, \tuN)\setminus\{0\}$,
    \item $(X_2, Y_2)\in \Adj(\J{E'}{P'},
    \tuN')\setminus\{0\}$, and 
    \item $(X_3, Y_3)\in \Adj(\J{E'}{P'},
    \J{\sigma(E)}{\sigma(P)}(\phi\bm{y}))\setminus\{0\}$.
\end{itemize}
Set, moreover, $\alpha= \diag( X_2X_3^{-1}\sigma X_1^{-1}, Y_2^{-1}Y_3\sigma Y_1) $ and $\beta=(Z')^{-1}\phi^{-1} \sigma Z$. 
Now, use \Cref{prop:autotopism} to construct a generating set $\mathcal{X}$ for
    $\mathrm{Auto}_{\ff}(\J{E}{P})$ and write
    \begin{align*} 
        \mathcal{Y}_1 &= \left\{\left(\begin{pmatrix} 
            X_1\gamma X_1^{-1} & 0 \\ 0 & Y_1^{-1}\delta Y_1
        \end{pmatrix}, Z^{-1}\epsilon Z\right) ~\middle|~ (\gamma, \delta, \epsilon) \in \mathcal{X} \right\}. 
    \end{align*} 
    If $P \in E[2](\ff)$, then set 
    \begin{align}\label{eqn:GL2-twist} 
        \mathcal{Y}_2 &= \left\{\left(\begin{pmatrix} 
            aI_3 & bX_1Y_1 \\ cY_1^{-1}X_1^{-1} & dI_3
        \end{pmatrix}, (ad-bc)I_3 \right) ~\middle|~ \begin{pmatrix} a & b \\ c & d \end{pmatrix} \in \GL_2(\ff) \right\}.
    \end{align} 
    If $P\notin E[2](\ff)$, then set 
    \begin{align}
        \mathcal{Y}_2 &= \left\{\left(\begin{pmatrix} 
            0 & X_1X_0Y_1 \\ Y_1^{-1}X_0X_1^{-1} & 0
        \end{pmatrix}, Z^{-1}Z_0Z \right) ~\middle|~ (X_0, Z_0) \text{ as in \Cref{lem:swap}}\right\}.
    \end{align} 
    Return the generating set $\mathcal{Y}=\mathcal{Y}_1\cup\mathcal{Y}_2$ for
    $\pseudo_{\ff}(\tuB)$ and $(\alpha, \beta) \in \pseudo_{L}(\tuB, \tuB')$. 
    
    \emph{Correctness.} For the existence of the matrices $Z$ and $Z'$ use \cref{rmk:flex}. Assume there exists $P\in E(\ff)$ such
    that $[\tuN] = [\J{E}{P}]$. By definition, there exist $X, Y\in \GL_3(\ff)$
    such that $X^{\tp}\tuN Y = \J{E}{P}$. Thus, $(A, B)\in \Adj(\tuN, \J{E}{P})$
    if and only if $(AX, Y^{-1}B) \in \Adj(\J{E}{P})$.  By \Cref{prop:JP-adj},
    the latter is $1$-dimensional and all non-zero elements are invertible.
    Thus, $\Adj(\tuN, \J{E}{P})$ is non-trivial if and only if $[\tuN]=[\J{E}{P}]$. Since $\det(\tuN) = 0$ defines a short Weierstrass equation
    and $\charac(\ff)\geq 5$, \cref{thm:ishitsuka} ensures the existence and
    uniqueness of such a $P$. A similar argument holds for $\tuN'$ and
    $\J{E'}{P'}$.

    If there is no isomorphism of $\phi: E'\to \sigma(E)$ mapping $P'$ to
    $\sigma(P)$, then \Cref{mainthm:sat} guarantees that $\tuB$ and $\tuB'$ are
    not $L$-pseudo-isometric. Otherwise, the isomorphism is linear being
    represented by a matrix in $\GL_3(\ff)$. Hence in this case,
    $\J{E'}{P'}(\phi\bm{y})$ is equivalent to $\J{\sigma(E)}{\sigma(P)}$, and
    the adjoint algebra $\Adj(\J{E'}{P'}(\phi\bm{y}), \J{\sigma(E)}{\sigma(P)})$
    is $1$-dimensional by a similar argument as before. \cref{prop:autotopism}
    and direct calculations show that $\mathcal{Y}\subset \pseudo_{\ff}(\tuB)$
    and $(\alpha, \beta)\in \pseudo_L(\tuB, \tuB')$. That $\langle \mathcal{Y}
    \rangle = \pseudo_{\ff}(\tuB)$ follows from the proof of \Cref{thm:pseudo}. 

    \emph{Complexity.} The complexity is dominated by listing and finding the
    unique points $P$ and $P'$. There are at most $O(q)$ points, and listing the
    points on an elliptic curve can be done using $O(q)$ field operations. 
\end{proof}

\begin{proof}[Proof of \Cref{mainthm:iso}\ref{thm-part:iso}]
    We assume the algorithm for \Cref{mainthm:iso}\ref{thm-part:recog1} has
    already been carried out successfully. Thus, the matrices $\tuB_1$ and
    $\tuB_2$ associated to $t_{G_1}$ and $t_{G_2}$ are written over their
    centroids, $\ff$, and are decomposable. Now use the algorithm in
    \cref{thm:computational-heart} with $L = \F_p$.
\end{proof}

\subsection{Implementation}
\label{sec:implementation}

We have implemented the algorithms in this section in the computer algebra
system \textsf{Magma}~\cite{magma/97}, and they are publicly
available~\cite{GitEGroups}. The plot in \cref{fig:data} shows the runtimes on
an Intel Xeon Gold 6138 2.00 GHz running \textsf{Magma} V2.26-11.

We describe the process shown in~\cref{fig:data}
from \cref{sec:main-results}. For each prime-power $q=p^e$ up to $10^5$,
avoiding integers with $p\in\{2,3\}$, we construct an elliptic curve in short
Weierstrass form $y^2 = x^3 + ax + b$, by choosing $(a,b)\in \F_q^2$ uniformly
at random and discarding any pair $(a,b)$ satisfying $4a^3 + 27b^2=0$. For each
elliptic curve $E$, we choose $P\in E(\F_q)\setminus \{\cor{O}\}$ uniformly at
random. After writing $\tuB_{E, P}$ over $\F_p$, our tensor is still represented
with convenient choices of bases. In order to remove this bias, we randomly
construct $X\in\GL_{6e}(\F_p)$ and $Z\in\GL_{3e}(\F_p)$, and set $\tuB = X^{\tp}
\tuB_{E,P}(Z\bm{y}) X$. We finally construct generators for the automorphism
group of $\GP_{\tuB}(\F_p)$ using \cref{mainthm:iso}.

\cref{mainthm:sat} gives us a characterization of the isomorphism classes of
elliptic $p$-groups for $p\notin \{2,3\}$. Such a characterization lends itself
more easily to explicit computations. In \cref{tab:iso-classes}, for each
prime-power $q\in [5, 97]$, with $q$ not equal to $2^e$ or $3^e$, we compute the
number of isomorphism classes, denoted by $N_q$, of the $G = \GP_{E, P}(\ff)$
such that $|G| = q^9$. 

\begin{table}[h]
    \centering 
    \begin{tabular}{|c|c||c|c||c|c||c|c||c|c|}
        \hline
        $q$ & $N_q$ & $q$ & $N_q$ & $q$ & $N_q$ & $q$ & $N_q$ & $q$ & $N_q$  \\ \hline \hline 
        5 & 31 & 19 & 381 & 37 & 1409 & 53 & 2863 & 73 & 5405 \\
        7 & 57 & 23 & 551 & 41 & 1723 & 59 & 3539 & 79 & 6321 \\
        11 & 131 & 25 & 385 & 43 & 1893 & 61 & 3785 & 83 & 6971 \\
        13 & 185 & 29 & 871 & 47 & 2255 & 67 & 4557 & 89 & 8011 \\ 
        17 & 307 & 31 & 993 & 49 & 1393 & 71 & 5111 & 97 & 9509 \\ \hline 
    \end{tabular}
    \smallskip 
    \caption{The number of isomorphism classes $N_q$ of the $\GP_{E, P}(\ff)$.}
    \label{tab:iso-classes}
\end{table}

\noindent The data in \cref{tab:iso-classes} provides good evidence that the
following conjecture seems true; in particular this would imply that the function $p\mapsto
N_p$ is quasipolynomial.
\begin{conj}\label{conj:N_p}
    For primes $p\geq 5$, we have 
    \[ 
        N_p = p^2 + p - \mathrm{gcd}(p - 2, 3) + \mathrm{gcd}(p - 1, 4).
    \] 
\end{conj}
\noindent It is clear that \Cref{conj:N_p} cannot be true for prime-powers: for
example,
\[ 
    N_{25} = 385 \neq 653 = 25^2 + 25 - \mathrm{gcd}(25 - 2, 3) + \mathrm{gcd}(25 - 1, 4).
\]
\begin{quest}
  For fixed $e$, is $p\mapsto N_{p^e}$ a quasipolynomial? 
\end{quest}

It may seem that the tensors we consider, namely $\tuB_{E, P}$, are somewhat
rare in general. Indeed, the existence of $3$-dimensional totally-isotropic
subspaces of $\ff^6$ should not occur ``at random''. However, this is not the
case. Let $\tuB\in \Mat_6(\ff[y_1,y_2,y_3]_1)$ whose Pfaffian defines an
elliptic curve in $\P^2_{\ff}$. By \cref{cor:four-cases}, there are four
cases for $|\ti{\tuB}|$, namely $0$, $1$, $2$, and $q+1$. By
\cref{thm:reducible-adjoints}, if $|\ti{\tuB}|\geq 1$, then there are three
different $*$-semisimple types: $\mathbf{O}_1(\ff)$, $\mathbf{X}_1(\ff)$, and
$\mathbf{S}_2(\ff)$. From computer evidence, it seems that $\Adj(\tuB) \cong
\mathbf{U}_1(L)$, where $L/\ff$ is a quadratic extension, whenever
$\ti{\tuB}=\varnothing$.

\begin{figure}[h]
    \centering 
    \includegraphics{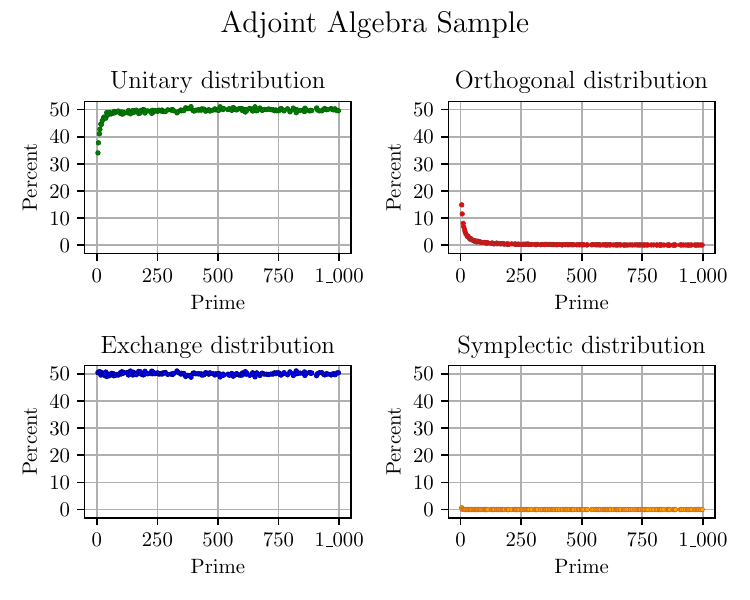}
    \caption{The isomorphism type for the $*$-semisimple part of the adjoint algebra for randomly constructed tensors.}
    \label{fig:adjoint}
\end{figure}

For primes $p\in [3, 1000]$, we constructed $1000$ matrices $\tuB \in
\Mat_{6}(\F_p[y_1, y_2, y_3]_1)$ where $\Pf(\tuB)$ defines an elliptic curve in
$\P^2_{\F_p}$. Each $\tuB$ was constructed uniformly at random: there are
$36$ homogeneous linear polynomials, so we chose $108$ elements from $\F_p$
uniformly at random to build $\tuB$, discarding matrices until they satisfied
the required condition. The outcome of the experiment is shown in
\cref{fig:adjoint}. It seems that, as $p\to \infty$, the probability of
$\Adj(\tuB)\cong \mathbf{X}_1(\ff)$ is equal to $0.5$, which seems to also be
equal to the probability of $\Adj(\tuB)\cong \mathbf{U}_1(L)$. In other words,
it seems that for ``random'' $\tuB \in \Mat_{6}(\F_p[y_1, y_2, y_3]_1)$, we have
either $|\ti{\tuB}|=2$ or $|\ti{\tuB}|=0$ at $50\%$ of the time. 

\bibliographystyle{abbrv}
\bibliography{e-groups}

\begin{thebibliography}{10}

\bibitem{Baer/38}
R.~Baer.
\newblock Groups with abelian central quotient group.
\newblock {\em Trans. Amer. Math. Soc.}, 44(3):357--386, 1938.

\bibitem{BanPal/12}
A.~Bandini and L.~Paladino.
\newblock Number fields generated by the 3-torsion points of an elliptic curve.
\newblock {\em Monatsh. Math.}, 168(2):157--181, 2012.

\bibitem{BMS/17}
M.~Bardestani, K.~Mallahi-Karai, and H.~Salmasian.
\newblock Kirillov's orbit method and polynomiality of the faithful dimension of {$p$}-groups.
\newblock {\em Compos. Math.}, 155(8):1618--1654, 2019.

\bibitem{Beauville00}
A.~{Beauville}.
\newblock {Determinantal hypersurfaces}.
\newblock {\em {Mich. Math. J.}}, 48:39--64, 2000.

\bibitem{magma/97}
W.~Bosma, J.~Cannon, and C.~Playoust.
\newblock The {M}agma algebra system. {I}. {T}he user language.
\newblock volume~24, pages 235--265. 1997.
\newblock Computational algebra and number theory (London, 1993).

\bibitem{BostonIsaacs04}
N.~Boston and M.~I. Isaacs.
\newblock Class numbers of $p$-groups of a given order.
\newblock {\em J. Algebra}, 279(2):810--819, 2004.

\bibitem{BMW/17}
P.~A. Brooksbank, J.~Maglione, and J.~B. Wilson.
\newblock A fast isomorphism test for groups whose {L}ie algebra has genus 2.
\newblock {\em J. Algebra}, 473:545--590, 2017.

\bibitem{BMW/20}
P.~A. Brooksbank, J.~Maglione, and J.~B. Wilson.
\newblock Exact sequences of inner automorphisms of tensors.
\newblock {\em J. Algebra}, 545:43--63, 2020.

\bibitem{BOW/19}
P.~A. Brooksbank, E.~A. O'Brien, and J.~B. Wilson.
\newblock Testing isomorphism of graded algebras.
\newblock {\em Trans. Amer. Math. Soc.}, 372(11):8067--8090, 2019.

\bibitem{BrooksbankWilson/12}
P.~A. Brooksbank and J.~B. Wilson.
\newblock Computing isometry groups of {H}ermitian maps.
\newblock {\em Trans. Amer. Math. Soc.}, 364(4):1975--1996, 2012.

\bibitem{BrooksbankWilson/15}
P.~A. Brooksbank and J.~B. Wilson.
\newblock The module isomorphism problem reconsidered.
\newblock {\em J. Algebra}, 421:541--559, 2015.

\bibitem{CH:isomorphism}
J.~J. Cannon and D.~F. Holt.
\newblock Automorphism group computation and isomorphism testing in finite groups.
\newblock {\em J. Symbolic Comput.}, 35(3):241--267, 2003.

\bibitem{CLO}
D.~A. Cox, J.~Little, and D.~O'Shea.
\newblock {\em Ideals, varieties, and algorithms}.
\newblock Undergraduate Texts in Mathematics. Springer, Cham, fourth edition, 2015.
\newblock An introduction to computational algebraic geometry and commutative algebra.

\bibitem{Cremona/03}
J.~Cremona.
\newblock {G1CRPC}: Rational points on curves, 2003.
\newblock \url{https://www.cise.ufl.edu/research/SpaceTimeUncertainty/Spatial3D/crem03.pdf}.

\bibitem{duSautoy/01}
M.~du~Sautoy.
\newblock A nilpotent group and its elliptic curve: non-uniformity of local zeta functions of groups.
\newblock {\em Israel J. Math.}, 126:269--288, 2001.

\bibitem{duSautoy/02}
M.~du~Sautoy.
\newblock Counting subgroups in nilpotent groups and points on elliptic curves.
\newblock {\em J. Reine Angew. Math.}, 549:1--21, 2002.

\bibitem{dS+VL}
M.~P.~F. du~Sautoy and M.~Vaughan-Lee.
\newblock Non-{PORC} behaviour of a class of descendant {$p$}-groups.
\newblock {\em J. Algebra}, 361:287--312, 2012.

\bibitem{Duke}
W.~Duke.
\newblock Elliptic curves with no exceptional primes.
\newblock {\em C. R. Acad. Sci. Paris S\'{e}r. I Math.}, 325(8):813--818, 1997.

\bibitem{ELGO:pgrp-automorphism}
B.~Eick, C.~R. Leedham-Green, and E.~A. O'Brien.
\newblock Constructing automorphism groups of {$p$}-groups.
\newblock {\em Comm. Algebra}, 30(5):2271--2295, 2002.

\bibitem{Fulton69}
W.~Fulton.
\newblock {\em Algebraic curves. {A}n introduction to algebraic geometry}.
\newblock W. A. Benjamin, Inc., New York-Amsterdam, 1969.

\bibitem{GS/84}
F.~Grunewald and D.~Segal.
\newblock Reflections on the classification of torsion-free nilpotent groups.
\newblock In {\em Group theory}, pages 121--158. Academic Press, London, 1984.

\bibitem{Hall}
P.~Hall.
\newblock The classification of prime-power groups.
\newblock {\em J. Reine Angew. Math.}, 182:130--141, 1940.

\bibitem{Hartshorne}
R.~Hartshorne.
\newblock {\em Algebraic geometry}.
\newblock Graduate Texts in Mathematics, No. 52. Springer-Verlag, New York-Heidelberg, 1977.

\bibitem{Hesse/44}
O.~Hesse.
\newblock \"{U}ber die {E}limination der {V}ariabeln aus drei algebraischen {G}leichungen vom zweiten {G}rade mit zwei {V}ariabeln.
\newblock {\em J. Reine Angew. Math.}, 28:68--96, 1844.

\bibitem{HIgman1}
G.~Higman.
\newblock Enumerating {$p$}-groups. {I}. {I}nequalities.
\newblock {\em Proc. London Math. Soc. (3)}, 10:24--30, 1960.

\bibitem{Higman2}
G.~Higman.
\newblock Enumerating {$p$}-groups. {II}. problems whose solution is porc.
\newblock {\em Proc. London Math. Soc. (3)}, 10:566--582, 1960.

\bibitem{Ishitsuka/17}
Y.~Ishitsuka.
\newblock A positive proportion of cubic curves over {$\Bbb Q$} admit linear determinantal representations.
\newblock {\em J. Ramanujan Math. Soc.}, 32(3):231--257, 2017.

\bibitem{IvanyosQiao/19}
G.~Ivanyos and Y.~Qiao.
\newblock Algorithms based on {$*$}-algebras, and their applications to isomorphism of polynomials with one secret, group isomorphism, and polynomial identity testing.
\newblock {\em SIAM J. Comput.}, 48(3):926--963, 2019.

\bibitem{KantorLuks/90}
W.~M. Kantor and E.~M. Luks.
\newblock Computing in quotient groups.
\newblock In {\em Proceedings of the Twenty-Second Annual ACM Symposium on Theory of Computing}, STOC '90, pages 524--534, New York, NY, USA, 1990. Association for Computing Machinery.

\bibitem{Lee/16}
S.~Lee.
\newblock A class of descendant {$p$}-groups of order {$p^9$} and {H}igman's {PORC} conjecture.
\newblock {\em J. Algebra}, 468:440--447, 2016.

\bibitem{Leedham-GreenSoicher/98}
C.~R. Leedham-Green and L.~H. Soicher.
\newblock Symbolic collection using {D}eep {T}hought.
\newblock {\em LMS J. Comput. Math.}, 1:9--24, 1998.

\bibitem{Luks/92}
E.~M. Luks.
\newblock Computing in solvable matrix groups.
\newblock In {\em 33rd Annual Symposium on Foundations of Computer Science, Pittsburgh, Pennsylvania, USA, 24-27 October 1992}, pages 111--120. {IEEE} Computer Society, 1992.

\bibitem{M17}
J.~Maglione.
\newblock Efficient characteristic refinements for finite groups.
\newblock {\em J. Symbolic Comput.}, 80(part 2):511--520, 2017.

\bibitem{M21}
J.~Maglione.
\newblock Filters compatible with isomorphism testing.
\newblock {\em J. Pure Appl. Algebra}, 225(3):Paper No. 106528, 28, 2021.

\bibitem{GitEGroups}
J.~Maglione.
\newblock \textsf{EGroups}, ver.~1.0, 2022.
\newblock \texttt{\url{https://github.com/joshmaglione/egroups}}.

\bibitem{Miller/78}
G.~L. Miller.
\newblock On the {{\(n\log{n}\)}} isomorphism technique (preliminary report).
\newblock In {\em Proceedings of the 10th annual ACM symposium on theory of computing, STOC'78, San Diego, CA, USA, May 1--3, 1978}, pages 51--58. New York, NY: Association for Computing Machinery (ACM), 1978.

\bibitem{Neumann/86}
P.~M. Neumann.
\newblock Some algorithms for computing with finite permutation groups.
\newblock In {\em Proceedings of groups---{S}t. {A}ndrews 1985}, volume 121 of {\em London Math. Soc. Lecture Note Ser.}, pages 59--92. Cambridge Univ. Press, Cambridge, 1986.

\bibitem{Ng}
K.~O. Ng.
\newblock The classification of {$(3,3,3)$} trilinear forms.
\newblock {\em J. Reine Angew. Math.}, 468:49--75, 1995.

\bibitem{O'BrienVoll/15}
E.~A. O'Brien and C.~Voll.
\newblock Enumerating classes and characters of {$p$}-groups.
\newblock {\em Trans. Amer. Math. Soc.}, 367(11):7775--7796, 2015.

\bibitem{RaviTri14}
G.~V. Ravindra and A.~Tripathi.
\newblock Torsion points and matrices defining elliptic curves.
\newblock {\em Internat. J. Algebra Comput.}, 24(6):879--891, 2014.

\bibitem{Rossmann/19}
T.~Rossmann.
\newblock The average size of the kernel of a matrix and orbits of linear groups, {II}: duality.
\newblock {\em J. Pure Appl. Algebra}, 224(4):106203, 28, 2020.

\bibitem{Rossmann/22}
T.~Rossmann.
\newblock On the enumeration of orbits of unipotent groups over finite fields.
\newblock {\em Proc. Amer. Math. Soc.}, 153(2):479--495, 2025.

\bibitem{RV/2019}
T.~Rossmann and C.~Voll.
\newblock Groups, graphs, and hypergraphs: average sizes of kernels of generic matrices with support constraints.
\newblock {\em Mem. Amer. Math. Soc.}, 294(1465):v+120, 2024.

\bibitem{Seress/03}
A.~Seress.
\newblock {\em Permutation group algorithms}, volume 152 of {\em Cambridge Tracts in Mathematics}.
\newblock Cambridge University Press, Cambridge, 2003.

\bibitem{Shparlinski}
I.~Shparlinski.
\newblock On finding primitive roots in finite fields.
\newblock {\em Theoret. Comput. Sci.}, 157(2):273--275, 1996.

\bibitem{silverman}
J.~H. Silverman.
\newblock {\em The arithmetic of elliptic curves}, volume 106 of {\em Graduate Texts in Mathematics}.
\newblock Springer, Dordrecht, second edition, 2009.

\bibitem{SilvermanTate}
J.~H. Silverman and J.~T. Tate.
\newblock {\em Rational points on elliptic curves}.
\newblock Undergraduate Texts in Mathematics. Springer, Cham, second edition, 2015.

\bibitem{stanojkovski2019hessian}
M.~Stanojkovski and C.~Voll.
\newblock Hessian matrices, automorphisms of {$p$}-groups, and torsion points of elliptic curves.
\newblock {\em Math. Ann.}, 381(1-2):593--629, 2021.

\bibitem{SageMath}
{The Sage Developers}.
\newblock {\em {S}ageMath, the {S}age {M}athematics {S}oftware {S}ystem ({V}ersion~8.7)}, 2019.
\newblock See {\url{https://www.sagemath.org}}.

\bibitem{VL/18}
M.~Vaughan-Lee.
\newblock Non-{PORC} behaviour in groups of order {$p^7$}.
\newblock {\em J. Algebra}, 500:30--45, 2018.

\bibitem{Voll/04}
C.~Voll.
\newblock {Zeta functions of groups and enumeration in Bruhat-Tits buildings}.
\newblock {\em Amer. J. Math.}, 126:1005--1032, 2004.

\bibitem{Voll/05}
C.~Voll.
\newblock {Functional equations for local normal zeta functions of nilpotent groups}.
\newblock {\em Geom. Func. Anal. (GAFA)}, 15:274--295, 2005.
\newblock with an appendix by A. Beauville.

\bibitem{vzGG/13}
J.~von~zur Gathen and J.~Gerhard.
\newblock {\em Modern computer algebra}.
\newblock Cambridge University Press, Cambridge, third edition, 2013.

\bibitem{Weinstein/16}
J.~Weinstein.
\newblock Reciprocity laws and {G}alois representations: recent breakthroughs.
\newblock {\em Bull. Amer. Math. Soc. (N.S.)}, 53(1):1--39, 2016.

\bibitem{Wilson/12}
J.~B. Wilson.
\newblock Existence, algorithms, and asymptotics of direct product decompositions, {I}.
\newblock {\em Groups Complex. Cryptol.}, 4(1):33--72, 2012.

\bibitem{Wilson/13}
J.~B. Wilson.
\newblock More characteristic subgroups, {L}ie rings, and isomorphism tests for {$p$}-groups.
\newblock {\em J. Group Theory}, 16(6):875--897, 2013.

\bibitem{Wilson/17}
J.~B. Wilson.
\newblock On automorphisms of groups, rings, and algebras.
\newblock {\em Comm. Algebra}, 45(4):1452--1478, 2017.

\bibitem{Wyman/72}
B.~F. Wyman.
\newblock What is a reciprocity law?
\newblock {\em Amer. Math. Monthly}, 79:571--586; correction, ibid. 80 (1973), 281, 1972.

\end{thebibliography}

\end{document}